\documentclass[aos]{imsart}

\RequirePackage{amsthm,amsmath,amsfonts,amssymb}
\RequirePackage[authoryear]{natbib}
\RequirePackage[colorlinks,citecolor=blue,urlcolor=blue]{hyperref}
\RequirePackage{graphicx}

\startlocaldefs
\theoremstyle{plain}
\newtheorem{theorem}{Theorem}
\newtheorem{lemma}[theorem]{Lemma}

\newtheorem{corollary}[theorem]{Corollary}
\theoremstyle{remark}
\newtheorem{remark}{Remark}
\newtheorem{assumption}[remark]{Assumption}
\newtheorem{condition}[remark]{Condition}
\newtheorem{example}[remark]{Example}
\newtheorem{definition}[remark]{Definition}

\newcommand{\s}[1][1]{\hspace{#1pt}}

\newcommand{\tq}[1]{{\textquotedblleft #1\textquotedblright}}

\newcommand{\oprst}{\textup} 

\newcommand{\var}{\textup{\oprst{var}}}

\newcommand{\Exp}{\textup{\oprst{E}}\s[.5]}
\newcommand{\Prb}{\textup{\oprst{P}}\s[.5]}

\newcommand{\tsum}{{\textstyle\sum\nolimits}}

\newcommand{\cst}{ \s[0.5] : \s[0.5] }

\newcommand{\ip}[2]{\langle #1, #2 \rangle}

\newcommand{\tr}{\textup{\text{tr}}}   
\newcommand{\diag}{\text{diag}}

\usepackage{pgffor}

\foreach \x in {A,B,...,Z,a,b,...,z} 
{
  \expandafter\xdef\csname cal\x\endcsname{\noexpand 
	\ensuremath{\noexpand\mathcal{\x}}}
  \expandafter\xdef\csname scr\x\endcsname{\noexpand 
	\ensuremath{\noexpand\mathscr{\x}}}
  \expandafter\xdef\csname bb\x\endcsname{\noexpand 
	\ensuremath{\noexpand\mathbb{\x}}}
  \expandafter\xdef\csname rm\x\endcsname{\noexpand 
	\ensuremath{\noexpand\mathrm{\x}}}
  \expandafter\xdef\csname bf\x\endcsname{\noexpand 
	\ensuremath{\noexpand\mbf{\x}}}
}

\newcommand{\ep}{\epsilon}

\newcommand{\upto}{ \uparrow }

\ifxetex
\let\gamma\upgamma
\let\epsilon\upepsilon
\let\tau\uptau
\let\pi\uppi
\let\kappa\upkappa
\let\omega\upomega
\fi

\newcommand{\hSig}{\bm{\hat}{\Sigma}}
\newcommand{\bSig}{\Sigma}
\newcommand{\hw}{\hat{w}}
\newcommand{\hsig}{\hat{\sigma}}
\newcommand{\col}[1]{\textsc{col}{\s[0.5] (#1)}}
\newcommand{\np}{n/\s[-0.5] p}
\newcommand{\tru}{\scrV}

\newcommand{\sv}{\gamma}
\newcommand{\Sv}{\Gamma}

\newcommand{\limp}{\lim_{p \upto \infty}}
\newcommand{\lsup}{\varlimsup_{p \upto \infty}}
\newcommand{\linp}{\varliminf_{p \upto \infty}}

\newcommand{\snr}{{{\Psi}}}
\newcommand{\hl}{\textup{\textsc{hdlss}}}
\newcommand{\hd}{\hl}
\newcommand{\pca}{\textup{\textsc{pca}}}
\newcommand{\gps}{\textup{\textsc{gps}}}

\newcommand{\Pd}{\Pi}

\newcommand{\zv}{\zeta}

\newcommand{\jj}{j \s[-2] j}

\newcommand{\est}{\hsig}
\newcommand{\HS}{\ensuremath{\textup{\textsc{hs}}}}
\newcommand{\seig}{\scrs}
\newcommand{\qeig}{\calS}

\newcommand{\hs}[1]{\| \s #1 \s\|_\HS}

\newcommand{\X}{\calX}
\newcommand{\E}{\calE}

\newcommand{\Id}{I}
\newcommand{\Y}{Y}
\newcommand{\Sam}{S}

\newcommand{\cn}{g}
\newcommand{\Jc}{J}

\newcommand{\Ond}{\bbO_B}
\newcommand{\Xv}{M}
\newcommand{\f}{f}

\newcommand{\pu}{\ensuremath{p_{\max}}}

\newcommand{\Lw}{\calW}
\newcommand{\hz}{\varphi}

\newcommand{\cl}{c_1}
\newcommand{\cc}{c_0}
\newcommand{\hx}{\hat{x}}
\newcommand{\nz}{\mu}
\newcommand{\hn}{\hat{\nz}}

\newcommand{\mf}{\hat{D}}
\newcommand{\rv}{r}

\newcommand{\Lp}{W_p}

\newcommand{\ran}{\scrM}

\newcommand{\zpH}{z_{\perp \scrH}}
\newcommand{\Rv}{\scrM}

\usepackage[scr=boondox]{mathalfa}
\usepackage{enumitem}
\usepackage{bm}
\usepackage{comment}
\usepackage{booktabs}
\newcommand{\ci}{\citet*}
\newcommand{\req}[1]{(\ref{#1})}

\endlocaldefs

\begin{document}

\begin{frontmatter}
\title{The Quadratic Optimization Bias of Large Covariance Matrices}
\runtitle{The Quadratic Optimization Bias of Large Covariance Matrices}

\begin{aug}
\author[A]{\fnms{Hubeyb}~\snm{Gurdogan}\ead[label=e1]{hgurdogan@math.ucla.edu}}
\and
\author[B]{\fnms{Alex}~\snm{Shkolnik}\ead[label=e2]{shkolnik@ucsb.edu}}
\address[A]{Department of Mathematics, University of California, Los Angeles, CA. \printead[presep={,\ }]{e1}}
\address[B]{Department of Statistics and Applied Probability, University of California, Santa Barbara, CA. \printead[presep={,\ }]{e2}}
\end{aug}

\begin{abstract}
We describe a puzzle involving  the interactions between an
optimization of a multivariate quadratic function and a
\tq{plug-in} estimator of a spiked covariance matrix.  When the
largest eigenvalues (i.e., the spikes) diverge with the
dimension, the gap between the true and the out-of-sample optima
typically also diverges. We show how to \tq{fine-tune} the
plug-in estimator in a precise way to avoid this outcome.
Central to our description is a \tq{\it quadratic optimization
bias} function, the roots of which determine this fine-tuning
property. We derive an estimator of this root from a finite
number of observations of a high dimensional vector.  This leads
to a new covariance  estimator designed specifically for
applications involving quadratic optimization. Our theoretical
results have further implications for improving low dimensional
representations of data, and principal component analysis in
particular.
\end{abstract}



\end{frontmatter}


\section{Introduction}
\label{sec:intro}

Optimization with a \tq{plug-in} model as an ingredient is
routine practice in modern statistical problems in engineering
and the sciences.  Yet the interactions between the optimization
procedure and the errors in an estimated model are often not
well understood.  Natural questions in this context include the
following: \tq{\it Does the optimizer amplify or reduce the
statistical errors in the model?  How does one leverage that
information if it is known? Which components of the model should
be estimated more precisely, and which can afford less
accuracy?} We explore these questions for the optimization of a
multivariate quadratic function that is specified in terms of a
large covariance model.  This setup is canonical for many
problems that are encountered in the areas of finance, signal-noise
processing, operations research and statistics.

Large covariance estimation occupies an important place in
high-dimensional statistics and is fundamental to multivariate
data analysis (e.g., \ci{yao2015}, \ci{fan2016a} and
\ci{lam2020}).  A covariance model generalizes the classical
setting of independence by introducing pairwise correlations.  A
parsimonious way to prescribe such correlations for many
variables is through the use of a relatively small number of
factors, which are high-dimensional vectors that govern all or
most of the correlations in the observed data.  This leads to a
particular type of covariance matrix, a so called
\tq{spiked-model} in which a small number of (spiked)
eigenvalues separate themselves with a larger magnitude from the
remaining (bulk) spectrum \citep{wang2017}.  Imposing this
factor structure may also be viewed as a form of regularization
which replaces the problem of estimating $p^2$ unknown
parameters of a $p \times p$ covariance matrix $\bSig$ with the
estimation of a few \tq{structured} components of this matrix.
Determining the components that require the most attention in a
setting that entails optimization is a central motivation of our
work.

\subsection{Motivation} \label{sec:motive}
To motivate the study of the interplay between optimization and 
model estimation error, we consider a quadratic function in $p$ 
variables. Let,
\begin{align} \label{quad} \s[32] Q(x) = \cc  + \cl \s
\ip{x}{\zv} - \frac{1}{2} \s \ip{x}{\Sigma x}
\s[32] ( x \in \bbR^p ) 
\end{align}
for an inner product $\ip{\s\cdot}{\cdot\s}$,  constants
$\cc,\cl \in \bbR$ and a vector $\zv \in \bbR^p$.  The $p \times
p$ matrix $\bSig$ is assumed to be symmetric and positive
definite. The maximization of $\req{quad}$ is encountered in
many classical contexts within statistics and probability
including least-squares regression, maximum a posteriori
estimation, saddle-point approximations, and Legendre-Fenchel
transforms in moderate/large deviations theory. Some related and
highly influential applications include Markowitz's portfolio
construction in finance \citep{markowitz1952},
Capon beamforming in signal processing \citep{capon1969} and
optimal fingerprinting in climate science \citep{hegerl1996}.
In optimization theory, quadratic functions form a key
ingredient for more general (black-box) minimization techniques
such as trust-region methods (e.g.,
\ci{maggiar2018}).\footnote{In this setting the covariance
matrix corresponds to an estimated Hessian matrix.} Since any
number of linear equality constraints may be put into the
unconstrained Lagrangian form in $\req{quad}$, our setting is
more general than it first appears.  Moreover, the maximization
of $\req{quad}$ is the starting point for numerous applications
of quadratic programming where nonlinear constraints are often
added.\footnote{To give one example, interpreting $\bSig$ as a
graph adjacency matrix and adding simple bound constants to
$\req{quad}$ leads to approximations of graph properties such as
the maximum independent set \citep{hager2015}. While $\bSig$ is
no longer interpreted as a spiked covariance matrix in a graph
theory setting, its mathematical properties are similar owing to
the celebrated Cheeger's inequality.}

The maximizer of $Q(\s \cdot \s)$ is given by 
$\cl \s \Sigma^{-1} \zv$ which attains the objective value
\begin{align} \label{conv}
\s[32]  \max_{x \in \bbR^p} Q(x) 
= \cc + \frac{\cl^2 \nz^2_p}{2}
\s[32] \big( \s \nz^2_p =  \ip{\zv}{\Sigma^{-1} \zv} \s \big) \s ,
\end{align}
but in practice, the maximization of  $Q(\s \cdot \s)$ is
performed with an  estimate $\hSig$ replacing the unknown
$\bSig$.  This \tq{plug-in} step is well known to yield a
perplexing problem  (see Section \ref{sec:lit}).  In essence,
the optimizer chases the errors in $\hSig$ to produce a
systematic bias in the computed maximum.  This bias is then
amplified by a higher dimension. 

Consider a high-dimensional limit $p \upto \infty$  and a
sequence of symmetric positive definite $\bSig = \bSig_{p \times p}$
with a fixed number $q$ of spiked eigenvalues diverging in
$p$ and all remaining eigenvalues bounded in $(0, \infty)$.  Let
$\hx$ be the maximizer of $\hat{Q}(\s \cdot \s)$, defined by
replacing $\bSig$ in $\req{quad}$ by estimates $\hSig = \hSig_{p
\times p}$ with the same eigenvalue properties. The estimated
objective is $\hat{Q}(\hx)$, but a more relevant quantity is the
realized objective, 
\begin{align} \label{Qhx}
Q(\hx) 
&= \cc + \cl \s \ip{\hx}{\zv} - \frac{1}{2} \ip{\hx}{\Sigma \hx}  
= \cc + \frac{\cl^2 \hn^2_p}{2} \s \mf_p
\end{align}
where $\hn^2_p = \ip{\zv}{\hSig^{-1} \zv}$ and $\mf_p$ is a
discrepancy (relative to $\req{conv}$) that can grow
rapidly as the dimension increases.  
Precluding edge cases where $\hn^2_p\s / \ip{\zv}{\zv}$
or $\ip{\zv}{\zv}$ vanish, unless $\hSig$ is 
fine-tuned in a calculated way, the following puzzling
behavior ensues.

\vspace{-0.04in}
\begin{enumerate}[leftmargin=0.32in, rightmargin=0.32in] {\it 
\item[~] The discrepancy $\mf_p$ tends to $-\infty$ as $p \upto \infty$
and consequently, the realized maximum
$Q(\hx)$ tends to $-\infty$ while the true
maximum $\req{conv}$ tends to $+\infty$. 
}\end{enumerate}
\vspace{-0.02in}

The asymptotic behavior above is fully determined by a certain
$\bbR^q$-valued function $\scrE_p(\s \cdot \s)$  which we derive
and call the \emph{quadratic optimization bias}. The way in
which this bias depends on the entries of $\hSig$ characterizes
the sought after interplay between the optimizer and the error
in the estimated covariance model. Mitigating the discrepancy
between the realized and true quadratic optima $\req{conv}$ and
$\req{Qhx}$ reduces to the problem of approximating the roots of
$\scrE_p(\s \cdot \s)$.  We remark that by parametrizing the
constants $\cc$ and $\cl$ in $p$, one can arrive at an
alternative limits for $\req{conv}$ and $\req{Qhx}$, but
practical scalings preserve the large disparity between the true
and realized objective values.  We examine (in Section
\ref{sec:risk}) the scaling $\cl = 1/p$ in particular, due to
its applicability to portfolio theory, robust beamforming and
optimal fingerprinting.

\subsection{Summary of results \& organization} \label{section:summary} 
The illustration above reflects that, in statistical settings,
solutions to estimated quadratic optimization problems exhibit
very poor properties out-of-sample.  Section \ref{sec:qob}
answers the question of which components of $\bSig$ must be
estimated accurately to reduce the discrepancy $\mf_p$ in
$\req{Qhx}$.  The size of $|\mf_p|$ is amplified by the growth
rate $\rv_p$ of the $q$ spiked eigenvalues, but is fully
determined by the precision of the estimate $\scrH$ of the
associated $p \times q$ matrix of eigenvectors $\scrB$ of
$\bSig$.  In particular, $\mf_p = -|\scrE_p(\scrH)|^2 \s O(r_p)$ 
where $\scrE_p(\scrH)$ is given by,
\begin{align} \label{obf} \s[32] 
\scrE_p(\scrH) = \frac{\scrB^\top z - (\scrB^\top \scrH)
(\scrH^\top z) }{\sqrt{1 - |\scrH^\top z|^2}} 
\s[32] \Big( z = \frac{\zv}{|\zv|} \Big) \s ,
\end{align}
for the Euclidean length $| \cdot |\s$.
Theorem~\ref{thm:discrepancy} gives sharp asymptotics for
$\mf_p$ in $\scrE_p(\scrH)$ and the other estimates/parameters.
Remarkably, the accuracy of the estimates of eigenvalues of
$\bSig$ is secondary relative the ensuring that $\scrH$ is such
that $\scrE_p(\scrH)$ is small for large $p$.  This is
noteworthy in view of the large literature on bias correction of
sample eigenvalues (or \tq{eigenvalue shrinkage}: see
\ci{ollila2020}, \ci{ledoit2021}, \ci{ledoit2022} and
\ci{donoho2023} for a sampling of recent work).  Instead, for
the discrepancy $\mf_p$, the estimation of the eigenvectors of
the spikes is what matters most.  We remark that while $\scrH =
\scrB$ forms a root of the map $\scrE_p \cst \bbR^{p \times q}
\to \bbR^q$ (i.e., $\scrE_p(\scrB) = 0_q$), it is not the only
root.  We refer to $\scrE_p (\s \cdot \s)$ as the quadratic
optimization bias (function) which was first identified in
\ci{goldberg2022} in the context of portfolio theory and for the
special covariance $\bSig$ with a single spike $(q = 1)$ and
identical remaining eigenvalues.\footnote{We state a more
general definition in Section \ref{sec:qob}, but $\scrH$ must
have orthonormal columns in $\req{obf}$. \label{fn:obf_orth}}

Section \ref{sec:pca} considers a $p \times p$ sample covariance
matrix $\Sam$ and its spectral decomposition
$\Sam = \scrH \qeig_p^2 \scrH^\top + G$, for a diagonal
$q \times q$ matrix
of eigenvalues $\qeig^2_p$, the associated $p \times q$ matrix
$\scrH$ of eigenvectors ($\scrH^\top \scrH = \Id_q$) and a
residual $G$. It is assumed that $\rv_p$ is $O(p)$ and that the
sequence $\Sam = \Sam_{p \times p}$ is based on a fixed number
of observations of a high dimensional vector.  Our
Theorem~\ref{thm:pcabias} proves that $\scrE_p(\scrH)$ is almost
surely bounded away from zero (in $\bbR^q$) eventually in $p$.
This has material implications for the use of principal
component analysis for problems motivated by
Section~\ref{sec:motive}.

Section \ref{sec:Hsharp} develops the following correction to
the sample eigenvectors $\scrH$. For the $q \times q$ diagonal
matrix $\snr$ satisfying $\snr^2 = \Id_q -\s  \tr(G)\s
\qeig_p^{-2} / n_q$ for $n_q \ge 1$, the difference between the
number of nonzero sample eigenvalues and $q$, we compute
\begin{align} \label{Hcorr}
\scrH \Psi + \frac{ z - \scrH \scrH^\top z }
{1 - |\scrH^\top z|^2} \s z^\top \scrH(\Psi^{-1}-\Psi ) \s .
\end{align}
Theorem \ref{thm:Hsharp} proves the $p \times q$ 
matrix of left singular vectors of
$\req{Hcorr}$, denoted $\scrH_\sharp$, has
\begin{align} \label{EHslim}
\s[32] \scrE_p(\scrH_\sharp) \to 0_q 
\s[32] (p \upto \infty) \s ,
\end{align}
almost surely.  The matrix $\scrH_\sharp$ constitutes a set of
corrected principal component loadings and is the basis of our
covariance estimator $\hSig_\sharp$. This matrix, owing to
$\req{EHslim}$, yields an improved plug-in estimator $x_\sharp =
\cl \hSig_\sharp^{-1} \zv$ for the maximizer of $\req{quad}$.
Thus, our work also has implications for the estimation of the
precision matrix $\bSig^{-1}$.  Theorem~\ref{thm:Hsharp} also
proves that the columns of $\scrH_\sharp$ have a larger
projection (than $\scrH$) onto the column space of $\scrB$.
Recent literature has remarked on the difficulty (or even
impossibility) of correcting such bias in eigenvectors (e.g.,
\ci{lw2017}, \ci{wang2017} and \ci{jung2022}).  That projection
is strictly better when $z$ in $\req{obf}$ has $|\scrB^\top z|$
bounded away from zero, i.e., captures information about that
subspace.  But $\req{EHslim}$ holds regardless, highlighting
that the choice of the \tq{loss} function (in our case
$\req{Qhx}$) matters.\footnote{See also \ci{donoho2018} for
another illustration of this phenomenon.} 

In Section \ref{sec:impossible}, we prove an impossibility
theorem (Theorem \ref{thm:angles}) that shows that without very
strong assumptions one cannot obtain an estimator of $\scrB^\top
\scrH$ asymptotically in the dimension if $q > 1$.  This has
negative implications for obtaining estimates of
$\scrE_p(\scrH)$ in $\req{obf}$ where $\scrB^\top \scrH$ is one
of the unknowns.  The latter contains all $q^2$ inner products
between the sample and population eigenvectors, and its
estimation from the observed data is an interesting theoretical
problem in its own right.  Our negative result adds to the
literature on high dimension and low sample size (\hl)
asymptotics, as inspired by \ci{hall2005} and
\ci{ahn2007}.\footnote{\ci{aoshima2018} survey much of the
literature since.} We remark that the {\hl} regime is highly
relevant for real-world data as a small sample size is often
imposed by experimental constraints, or by the lack of
long-range stationarity of time series. The content of Theorem
\ref{thm:angles} also points to a key feature that distinguishes
our work from \ci{goldberg2022}) who fix $q =1$. Another aspect
making our setting substantially more challenging is that
we find roots of a multivariate function $\scrE_p(\s \cdot
\s)$ (which is univariate when $q = 1$).

\begin{figure}[htp!]
\vspace{-0.1in}
\centering
\includegraphics[width=4in]{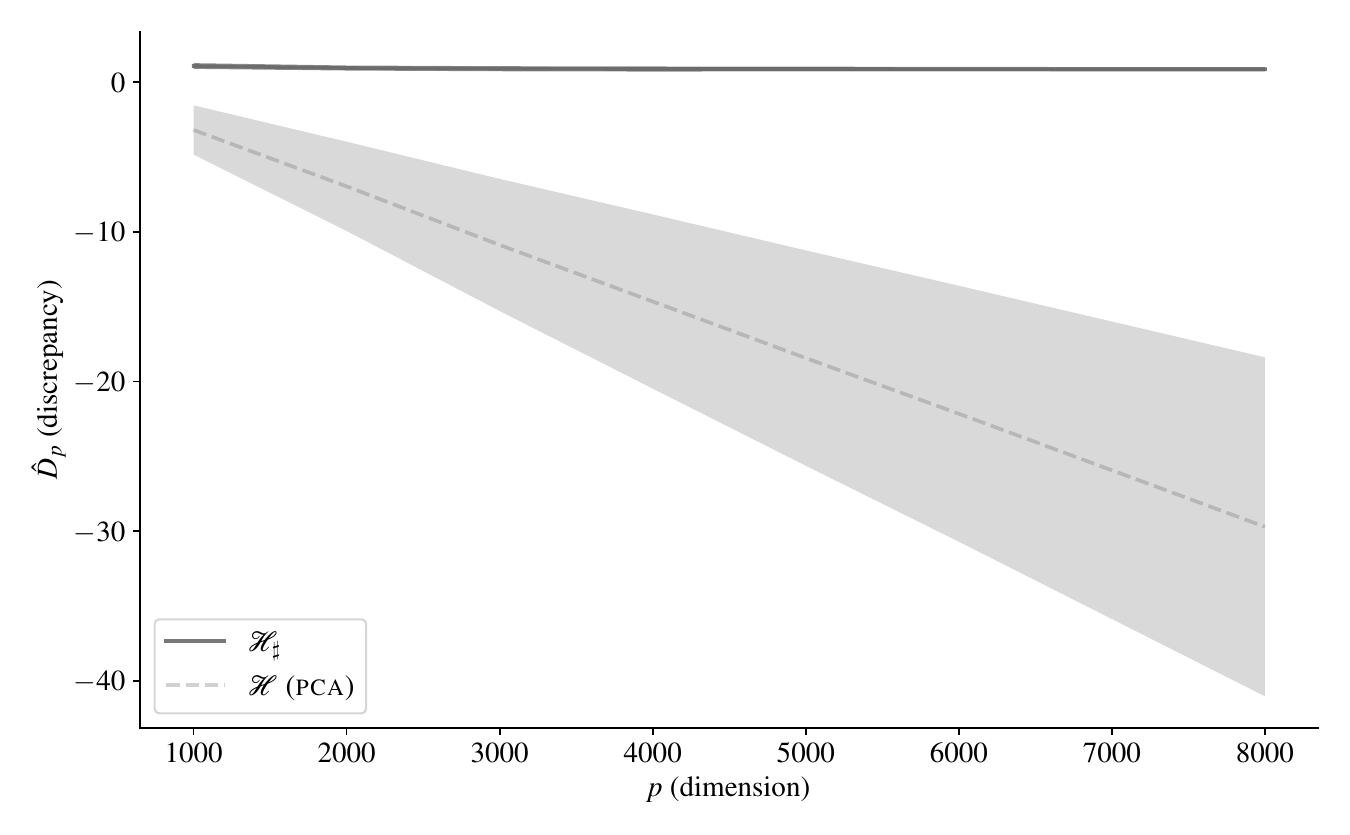}
\vspace{-0.1in}
\caption{Discrepancy $\mf_p$ (with two standard deviation error bars) 
for two covariance models estimated from the 
simulated data sets of Section \ref{sec:numerics}.
The first (solid line) is based on $\req{Hcorr}$ 
and the resulting corrected eigenvectors $\scrH_\sharp$ . The
second (dashed line) is based on the raw 
sample eigenvectors $\scrH$ (\pca). The optimal
$\mf_p$ equals $1$.}
\label{fig:Dp}
\end{figure}

In terms of applications, our results generalize those of
\ci{goldberg2022} to covariance models that hold wide acceptance
in the empirical literature on financial asset return (i.e., the
Arbitrage Pricing Theory of \ci{ross1976}, \ci{huberman1982},
\ci{chamroth1983} and others). Section \ref{sec:numerics}
investigates the problem of minimum variance investing with
numerical simulation, and demonstrates that the estimator
$\scrH_\sharp$ results in vanishing asymptotic portfolio risk
and a bounded discrepancy $\mf_p$ (see Figure~\ref{fig:Dp}).
Appendix \ref{app:lit} summarizes other applications including
signal-noise processing and climate science as related to
Section~\ref{sec:motive}.

\subsection{Limitations \& related literature}
\label{sec:lit}
Our findings in Section \ref{sec:motive} form a starting point
for important extensions and applications.  Extending the
estimator in $\req{Hcorr}$ to general quadratic programming with
inequality constraints would greatly expand its scope. In terms
of covariance models, we require spikes that diverge linearly
with the dimension, which excludes several alternative
frameworks in the literature.\footnote{This includes the
Johnstone spike model, in which all eigenvalues remain bounded
as the dimension grows, and its extensions (e.g.,
\ci{johnstone2001}, \ci{paul2007}, \ci{johnstone2009} and
\ci{bai2012}). Futher generalizations include slowly growing
spiked eigenvalue models as in \ci{de2008}, \ci{onatski2012},
\ci{shen2016} and \ci{bai2023}.} Likewise, the asymptotics of
the data matrix aspect ratio $p/n$ differs across applications.
We also do not address the important setting in which the number
of spikes $q$ is misspecified.\footnote{There is a large
literature on the estimation of the number of
spikes/factors/principal components.  Most relevant to our setup
(high dimension and low sample size) is \ci{jung2018}.} Finally,
the established convergence in $\req{EHslim}$ leaves the
question of rates  unanswered.  This is particularly important
for problems requiring the discrepancy $\mf_p$ to not grow too
quickly.  We offer no theoretical treatment of convergence rates
but our numerical results suggest this quantity remains bounded
as $p$ grows (c.f., Figure~\ref{fig:Dp}).

The work we build on directly was initiated in \ci{goldberg2022}.
We refer to their proposal as the GPS estimator and derive it in
Section~\ref{sec:gps}.  Important extensions are developed in
\ci{gurdogan2022} and \ci{goldberg2023}.  The GPS estimator was
shown to be mathematically equivalent to a James-Stein
estimation of the leading eigenvector of a sample covariance
matrix in \ci{shkolnik2022}.  These results share much in common
with the ideas found in \ci{casella1982}.  For a survey of the
above literature, focusing on connections to the James-Stein
estimator, see \ci{goldberg2023surv}.  The GPS estimator is
explained in terms of regularization in \ci{lee2024a}, and
\ci{lee2024b} derive central limit theorems for this estimator
as relevant for the convergence of the discrepancy $\mf_p$.
Some numerical exploration of the case of more than one spike is
found in \ci{goldberg2020}. 

The spiked covariance models we consider, and the application of
{\pca} for their estimation, are rooted in the literature on
approximate factor models and \tq{asymptotic principal
components} originating with \ci{chamroth1983} and
\ci{connor1986}. Recent work in this direction is well
represented by \ci{bai2008}, \ci{fan2013}, \ci{bai2023} and
\ci{fan2023}.  In this literature, the work that most closely
resembles ours, by focusing on improved estimation of sample
eigenvectors, is \ci{fan2016b}, \ci{fan2018} and
\ci{lettau2020}.  \ci{fan2016b} project the data onto a space
generated by some externally observed covariates, improving the
resulting sample eigenvectors when the covariates have
sufficient explanatory power. \ci{fan2018} apply a linear
transformation to the sample eigenvectors in an approach that is
most closely related to formula $\req{Hcorr}$. We also apply a
linear transformation, but the eigenspace is first augmented by
the vector $\zv$ in $\req{quad}$.\footnote{We remark that with a
single spike/factor (i.e., $q =1$), a linear transformation of
the eigenvector(s) adjustment only the eigenvalue, not the
eigenvector itself due to its unit length normalization. Further
differences with \ci{fan2016b} arise in the estimation of the
optimal linear transformation.} \ci{lettau2020} extract
principal components from a rank-one updated sample covariance
matrix. This update is based on insight from asset pricing
theory and it is unclear how the resulting sample eigenvectors
are related to formula $\req{Hcorr}$. The same applies to the
very closely related literature on sample covariance matrix
shrinkage (e.g., \ci{ledoit2004b}, \ci{fisher2011},
\ci{tomer2014} and \ci{wang2024}).\footnote{This takes the form
$\hSig = a \s \Sam + (1- a) F$ for some $a \in [0,1]$ and matrix
$F$.  Targets $F \neq \Id$ adjust eigenvectors but in ways that
may be difficult to quantify via closed-form expressions (c.f.
$\req{Hcorr}$).  \label{lwfoot}}

The vast majority of the literature on approximate factor models
and covariance estimation assumes the data matrix aspect ratio
$p/\s[-0.5] n$ tends to a finite constant
asymptotically.\footnote{This may be due to the outsized
influence of random matrix theory (e.g., \ci{mp1967}). Another
reason may be the consistency of the sample eigenvectors that
can be achieved in this regime (see \ci{yata2009},
\ci{shen2016gen} and \ci{wang2017}.} In contrast, our analysis
of a finite sample in the high dimensional limit draws on the work
on {\pca} in \ci{jung2009}, \ci{jung2012} and \ci{shen2016} and
others.  In the latter, the {\hl} asymptotics for the matrix
$\scrB^\top \scrH$, appearing in $\req{obf}$, have already been
worked out (but see Section~\ref{sec:impossible} for our
impossibility theorem). Our main focus is on correcting the
biases that the asymptotics of $\scrB^\top \scrH$ reveal.  For
approaches to correcting the finite sample bias in eigenvalues
and principal component scores, see \ci{yata2012},
\ci{yata2013}, \ci{jung2022} and our Remark \ref{rem:eigenvals}.
\ci{shen2013} apply regularization in the presence of sparsity
in the population eigenvectors to correct finite sample bias in
the principal components.  It is unclear how their estimators
are related to the update in $\req{Hcorr}$, but we do not impose
such sparsity assumptions.

Several other strands of the {\pca} literature are relevant as
their aims coincide with improved sample eigenvector estimation.
In one direction is the literature on sparse and low-rank matrix
decompositions (e.g. \ci{chandra2012}, \ci{saunderson2012},
\ci{bai2019}, \ci{farne2024} and \ci{li2024}).  These convex
relaxations aim to find more accurate low-dimensional
representations of the data and are sometime referred to as
forms of robust {\pca} \citep{candes2011}.  In a related direction
is the recent work on robust {\pca} for heteroskedastic noise
(e.g., \ci{cai2021}, \ci{zhang2022}, \ci{yan2021},
\ci{agterberg2022} and \ci{zhou2023}). These efforts provide
($p,n$ finite) bounds on generalized angles between the true and
the estimated subspaces and complement our asymptotic {\pca}
results in Sections \ref{sec:pca} \& \ref{sec:Hsharp}.
Perturbations of eigenvectors have also been recently revisited
in \ci{fan2018inf}, \ci{abbe2022} and \ci{li2022}. The latter
use these bounds to construct estimators that \tq{de-bias}
linear forms such as $\scrB^\top z$ appearing in $\req{obf}$.
These results can likely supply alternative proofs to ours (or
even convergence rates), but our focus is on limit theorems only.

Lastly, we emphasize the area of mean-variance portfolio
optimization. As the literature on this topic is quite vast, we
mention only a few strands related to Section~\ref{sec:motive}.
Examples of early influential work in this direction include
\ci{michaud1989} and \ci{best1991}. For numerical simulations
that illustrate the impact on practically motivated models and
metrics see \ci{bianchi2017}.  A random matrix theory
perspective on the behavior of objectives related to $\req{Qhx}$
may be found in \ci{pafka2003}, \ci{bai2009}, \ci{el2010},
\ci{el2013}, \ci{bun2017} and \ci{bodnar2022}.  Highly relevant
recent work in econometrics using latent factor models includes
\ci{ding2021} who consider a portfolio risk measure closely tied
to $\req{Qhx}$.  Bayesian approaches to mean-variance
optimization include \ci{lai2011} and \ci{bauder2021}. These
estimators are closely related to Ledoit-Wolf shrinkage
(\ci{ledoit2003} and \ci{ledoit2004a}) which itself has
undergone numerous improvements (e.g., \ci{lw2018} and
\ci{lw2020pow}). In tandem, shrinkage methods have been known to
impart effects akin to extra constraints in the portfolio
optimization as early as \ci{jm2003}. An insightful example of
such robust portfolio optimization that relates $\req{Qhx}$ to
the convergence of the covariance matrix estimator is developed
in \ci{fan2012}. More advanced robust portfolio optimizations
have also been proposed (e.g., \ci{boyd2024}). Alternatively,
constraints are often applied in the covariance matrix
estimation process as an optimization in itself. For example,
\ci{won2013} apply a condition number constraint that leads to
non-linear adjustments of sample eigenvalues (c.f.,
\ci{lw2020}), but leaves the sample eigenvectors unchanged.
\ci{challet2023} document the difficulty with relying solely on
eigenvalue correction, especially for small sample sizes.
\ci{cai2020} apply sparsity constraints (on the precision
matrix) and analyze optimality properties related to
$\req{Qhx}$.  We emphasize that the impact of such constraints
on eigenvectors is difficult (or impossible) to quantify, in
contrast to formula $\req{Hcorr}$.\footnote{It should be noted
that another interesting approach to mean-variance portfolio
optimization concerns the direct shrinkage of the portfolio
weights (i.e., akin to shrinkage $\hx$ in $\req{Qhx}$, e.g.,
\ci{bodnar2018}, \ci{bodnar2022} \ci{bodnar2023}). }

\subsection{Notation}  \label{sec:notation}
Let $\col{A}$ denote the column span of the matrix $A$ and
let $\zv_A$ denote the orthogonal projection of the vector
$\zv$ on $\col{A}$, e.g.,
\begin{align} \label{eH}
 \zv_A = AA^\dagger \zv
\end{align}
where $A^\dagger = (A^\top\s[-1] A)^{-1} A^\top$, 
the Moore-Penrose inverse of a full column rank matrix $A$.
We use $\Id$ to denote an identity matrix
and $\Id_q$ when highlighting its dimensions,
$q \times q$. 

Write $\ip{u}{v}$ for the scalar product of $u,v \in \bbR^m$,
let $|u| = \sqrt{\ip{u}{u}}$ and $|A|$ be the induced (spectral)
norm of a matrix  $A$. We denote by $\nu_{m \times q}(\s \cdot
\s)$ a function that given a $m \times \ell$ matrix $A$,
uniquely selects (see Appendix \ref{app:select}) an enumeration
of its singular values $|A| = \Lambda_{11} \ge \cdots \ge
\Lambda_{\min(\ell, m)}$ and outputs a $m \times q$ matrix
$\nu_{m \times q}(A)$ of left singular vectors with the values
$\Lambda_{11}, \dots, \Lambda_{qq}$ in columns $1, \dots, q \le
\min(m,\ell)$. That is,
\begin{align}  \label{select} 
\nu_{m \times q}(A) = \nu_{m \times q}(AA^\top) = \scrL 
\s[4] \cst \s[4]  \scrL^\top A = \Lambda \scrR^\top, \s[8] \scrL^\top \scrL = \Id_q
= \scrR^\top \scrR \s ,
\end{align}
where $\Lambda$ is the $q \times q$ diagonal with entries
$\Lambda_{11}, \dots, \Lambda_{qq}$, and $\scrR$
is the $\ell \times q$ matrix of right singular vectors of $A$.
The $m \times q$ matrix $\nu_{m \times q} (A)$
also corresponds to some
unique choice of eigenvectors of $AA^\top$ with $q$
largest eigenvalues
$\Lambda_{11}^2 \ge \cdots \ge \Lambda^2_{qq}$.

We take $0_q = (0, \dots, 0) \in \bbR^q$ and 
$1_q = (1, \dots, 1) \in \bbR^q$.
Lastly, $\linp$ and $\lsup$ denote
the limit inferior and superior as $p \upto \infty$, and 
$A = A_{p  \times m}$, a sequence of matrices
with dimensions $p \times m$ when at least one of $p$ or $m$
grows to infinity.

\section{Quadratic Optimization Bias}
\label{sec:qob}

We begin with a $p \times p$  covariance  matrix $\bSig$ which has the
decomposition,
\begin{align} \label{bSig}
 \Sigma = B B^\top + \Sv 
\end{align}
for a $p \times q$ full rank  matrix $B$ and some
$p\times p$ invertible, symmetric matrix $\Sv$.

The covariance decomposition in $\req{bSig}$ is often associated with
assuming a factor model (e.g., see \ci{fan2008}).  In the
context of large covariance matrix estimation, the following
approximate factor model framework is by now standard.\footnote{These
conditions originate with \ci{chamroth1983} and Assumption
\ref{asm:afm} closely mirrors theirs as well as those of later
work such  as \ci{fan2008} and \ci{fan2013}.}

\begin{assumption} \label{asm:afm}
The matrices $B = B_{p \times q}$ and $\Sv =\Sv_{p \times
p}$ satisfy the following.
\begin{enumerate}[label=\text{(\alph*}), itemsep=0in] 
\item $0 < \linp \inf_{|v|=1}\ip{v}{\Sv v} < 
\lsup \sup_{|v|=1}\ip{v}{\Sv v} < \infty$.
\label{afm:Omg}
\item $\limp \s (B^\top \s[-1] B)\s /\s[-1] p$ exists as an
invertible $q \times q$ matrix with fixed $q \ge 1$. \label{afm:BB}
\end{enumerate}
\end{assumption}


In the literature on factor analysis, the entries of a column of
$B$ are called loadings, or exposures to a risk factor
corresponding to that column.  Condition~\ref{afm:BB} of
Assumption \ref{asm:afm} states that all $q$ risk factors are
persistent as the dimension grows and implies the $q$ largest
eigenvalues of $\bSig = \bSig_{p \times p}$ grow linearly in
$p$.  Condition~\ref{afm:Omg} states that all remaining variance
(or risk) vanishes in the high dimensional limit and the bulk
($p-q$ smallest) eigenvalues of $\bSig_{p \times p}$ are bounded
in $(0,\infty)$ eventually. The  $\Sv$ matrix is associated with
covariances of idiosyncratic errors, but can have alternative
interpretation (e.g., covariance of the specific return of
financial assets).  Assumption~\ref{asm:afm} implies  $\limp
\scrB^\top \nu_{p \times q}(\bSig) \to \Id_q$ for the $p \times q$
sequence $\scrB = \nu_{p \times q}(B)$ of eigenvectors of
$BB^\top$ with nonzero eigenvalues.  The latter implication
motivates the frequent reference to the $\scrB = \scrB_{p \times q}$ as
the asymptotic principal components of $\bSig_{p \times p}$.

In practice, $\bSig$ is unknown and an estimated
model $\hSig$ is used instead. Let,
\begin{align} \label{hSig}
  \hSig = HH^\top +  \hat{\sv}^2 \Id \s . 
\end{align}
for a full rank $p \times q$ matrix $H$ and a number $\hat{\sv}
> 0$. We assume $q$ is known and allow for $\hat{\sv}$ to depend on $p$
provided this sequence  is bounded in $(0,\infty)$.  We do not
pursue alternative (to $\hat{\sv}^2 \Id$) estimates of $\Sv$
because, as pointed out below, accurate estimation of 
the matrix $\Sv$ is
of secondary concern relative to the accuracy of the estimate
$H$.

For $\zv \in \bbR^p$, the eigenvectors $\nu_{p \times q}(B) = \scrB$
and  $z_H = HH^\dagger z$ per $\req{eH}$, define
\begin{align} \label{optbias}
\s[32] \scrE_p (H) 
= \frac{ \scrB^\top (z-z_H)}{ |z - z_H|}
\s[32] \Big( z = \frac{\zv}{|\zv|}\Big)
\end{align}
assuming $|z -z_H| \neq 0$.
We note $|\scrE_p(H)|\le 1$ and that $\req{optbias}$ is a precursor 
to the
quadratic optimization bias function $\scrE_p(\s \cdot \s )$  in
$\req{obf}$, but the $H$ in $\req{optbias}$  need not have 
orthonormal columns (fn. \ref{fn:obf_orth}). These two 
definitions are equated in Section \ref{sec:pca}.

All results in this section continue to hold with
$\req{optbias}$ redefined with any $\scrB$ such that $BB^\top =
\scrB \Lambda_p^2 \scrB^\top$ with $|\scrB|$ bounded in $p$ and
diagonal $q \times q$  matrix $\Lambda^2_p$, not necessarily the
eigenvalues.  This alternative may be useful for some
applications.

\subsection{Discrepancy of quadratic optima in high dimensions} 
\label{sec:discrepancy}
Returning to the optimization setting of Section 
\ref{sec:motive}, for constants $\cc,\cl \in \bbR$ and 
$\zv \in \bbR^p$, we consider
\begin{align}  \label{hQx}
\hat{Q}(x) 
&= \cc + \cl \s \ip{x}{\zv} - \frac{1}{2} \ip{x}{\hSig x}  
\end{align}
which attains $\max_{x \in \bbR^p}\hat{Q}(x) = \hat{Q}(\hx)
= \cc + \frac{\cl^2 \hn_p^2}{2} $ at the maximizer
$\hx \in \bbR^p$ analogously to $\req{conv}$
but with $\hn_p^2 = \ip{\zv}{\hSig^{-1} \zv}$.
Because $\hat{Q}(\s \cdot \s)$ is not the 
true objective function  $Q (\s \cdot \s)$ in $\req{quad}$,
we are interested in the realized objective $Q(\hx)$. Now,
\begin{align} \label{QhxD}
Q(\hx) 
= \cc +  \frac{\cl^2 \hn^2_p}{2} \s 
\Big( 2 -  \frac{\ip{\hx}{\bSig \hx} }{ \cl^2 \hn_p^2} \Big) 
= \cc +  \frac{\cl^2 \hn^2_p}{2} \s \mf_p  \s , 
\end{align}
which identifies the 
discrepancy $\mf_p$ in $\req{Qhx}$
relative to both $\hat{Q}(\hx)$ and $\req{conv}$.

To avoid division by zero in $\req{optbias}$, we
prevent $\zv \in \bbR^p$ from vanishing and residing entirely 
in $\col{H}$ asymptotically 
(i.e., $|1 - z_H|^2 = 1 - |\zv_H|^2 / |\zv|^2$).\footnote{This
edge case must be treated separately from our analysis and we do not 
pursue it. The entries of $\zv$ may be viewed as the  first $p$ 
entries of an infinite sequence or as rows of a triangular array.} 
We further assume the estimate $H$ has properties consistent 
$B$ in view of Assumption \ref{asm:afm} \ref{afm:BB}.

\begin{assumption} \label{asm:eH} 
Suppose 
$H = H_{p \times q}$ and $\zv = \zv_{p \times 1}$ satisfy
$\lsup |\zv_H|\s / \s |\zv| < 1$ and
$\linp |\zv| \neq 0$. Also,
$\limp (H^\top H)\s /\s[-1] p$ exists as 
a $q \times q$ invertible matrix, 
\end{assumption}


We address the asymptotics of the discrepancy
$\mf_p = 2 -  \ip{\hx}{\bSig \hx}/(\cl \hn_p)^2$ in
$\req{QhxD}$,
letting $BB^\top = \scrB \Lambda_p^2 \scrB$ as above, with
$\scrB = \nu_{p \times q}(B)$ the canonical choice.

\begin{theorem} \label{thm:discrepancy}
Suppose Assumptions \ref{asm:afm} and \ref{asm:eH} hold.
Then, for $u_H = \frac{z- z_H}{|z-z_H|}$, 
\begin{align*} 
\mf_p 
= -    \frac{|  \Lambda_p  \s[2] \scrE_p(H)  |^2 }{\hat{\sv}^2}
+\Big( 2 - \frac{\ip{u_H}{\Sv u_H}}{ \hat{\sv}^2} \Big) 
+  O \Big(|\scrE_p(H)| +|\scrE_p(H)|^2
+ \frac{1}{p}  \Big) \s .
\end{align*}
\end{theorem}

\begin{remark} 
The proof (see Appendix \ref{app:qob}) has a more general
statement by relaxing the rate of growth of the eigenvalues of
$\Lambda^2_p$ to a sequence $\rv=\rv_p$ (rather than $p$). That
is, we only  assume the limits of $B^\top B/\rv_p$ and $H^\top
H/\rv_p$ are invertible matrices. In this case, $O(1/p)$ is
replaced by $O(1/\rv_p)$ above. This shows $|\mf_p|$ is in $
O(\rv_p \s |\scrE_p(H)|^2 )$. 
\end{remark}
Theorem \ref{thm:discrepancy} reveals that $\mf_p$ 
diverges to $-\infty$ unless we find roots  
of $\scrE_p(\s \cdot \s)$, perhaps asymptotically.
Note that $\scrE_p(B) =0$, but
other roots exists (see Section \ref{sec:Hsharp}). 


\begin{lemma} \label{lem:K}
 For any full rank $p \times q$ matrix $H$ with 
$|\zv_H| < |\zv|$ and  any $q \times q$ invertible matrix $K$, 
we have $\scrE_p(H) =\scrE_p(HK)$. 
\end{lemma}

\begin{proof} This follows by a direct verification
using the definition in $\req{eH}$. 
 \begin{align*}
\zv_{HK}  &= (HK)^\dagger \zv = 
 (HK) ( (HK)^\top (HK) )^{-1} (HK)^\top \zv \\
&= (HK) K^{-1} (H^\top H)^{-1} K^{-\top} (HK)^\top \zv 
\\&= H (H^\top H)^{-1} H^\top \zv = H^\dagger \zv = \zv_H
\end{align*}
 and with the definition of $\scrE_p(\s
\cdot \s)$ in $\req{optbias}$ we obtain the desired result.
\end{proof}

Lemma \ref{lem:K} pinpoints what constitutes a poor \tq{plug-in}
covariance estimator $\hSig$. For example, the column lengths of
$H$ have no effect on the quadratic optimization bias
$\scrE_p(H)$.  For the eigenvalue decomposition $HH^\top = \scrH
\qeig_p^2 \scrH^\top$ (with $K = \qeig_p^{-1}$ in Lemma
\ref{lem:K}), we see that $\scrE_p(H) = \scrE_p(\scrH)$.  Thus,
to fine-tune $\hSig$ for quadratic optimization, one need
correct only the basis $\col{H}$. This amounts to
finding the (asymptotic) roots of the function $\scrE_p(\s \cdot
\s)$.  If the convergence to a root is sufficiently rapid, one
may then estimate $\ip{u_H}{\Sv u_H}$ closely by $\hat{\sv}^2$
to bring the discrepancy $\mf_p$ to one per Theorem
\ref{thm:discrepancy}.  We conclude this section by showing that
for many applications the rate of convergence of $\scrE_p(H)$ is
less important than Theorem \ref{thm:discrepancy} suggests.

\subsection{Applications}
\label{sec:risk}
To illustrate some important examples in practice, we 
consider the following canonical, constrained 
optimization problem.
\begin{equation} \label{minvar}
\begin{aligned}
  \min_{w \in \bbR^p}  & \s[4] \est^2 \\
  \est^2 &= \ip{w}{\hSig w} \\
  \ip{w}{\zv} &= 1   
\end{aligned}
\end{equation}
Now, $\hat{Q}(\s \cdot \s)$  in $\req{hQx}$ is the Lagrangian
for $\req{minvar}$ with  $\cc = 0$ and $\cl =
\ip{\zv}{\hSig^{-1} \zv}^{-1}$, which decays as $1/p$ under
Assumption \ref{asm:eH}.  The minimizer $\hw \in \bbR^p$ of
$\req{minvar}$ corresponds to the weights of a minimum variance
portfolio of financial assets with $\zv = 1_p$ implementing the
\tq{full-investment} constraint. Minimum variance and the more
general mean-variance optimized portfolios are  widely used in
finance.  Here, the $p$ entries of a column of $B$ in
$\req{bSig}$ represent the exposures of $p$ assets to that
risk factor, e.g.,  market risk (bull/bear market), industry
risk (energy, automotive, etc.), climate risk (migration,
drought, etc.), innovation risk (Chat GPT, etc). Similar
formulations based on $\req{minvar}$ arise in signal-noise
processing and climate science (see Appendix
\ref{app:lit}).\footnote{In signal-noise processing,
$\req{minvar}$ maximizes the signal-to-noise ratio of a
beamformer with $\zv$ referred to as the \tq{steering vector}.
The same is done for optimal fingerprinting in climate science
with $\zv$ called the \tq{guess pattern}. We review this
literature with emphasis on estimation of $\bSig$ in
Appendix~\ref{app:lit}.}

Continuing with the above example, the minimum $\est^2$
of $\req{minvar}$  corresponds to the variance of the 
estimated portfolio $\hw$, while the expected out-of-sample variance
is,
\begin{align} \label{tru}
 \tru_p^2 = \ip{\hw}{\Sigma \hw}   \s .
\end{align}
We have $\mf_p = 2 - \hn_p^2 \tru_p^2$ (see Appendix
\ref{app:qob}) and, under the conditions of
Theorem~\ref{thm:discrepancy},
\begin{align} \label{tru_asymp}
 \tru_p^2 =
\frac{|\Lambda_p \scrE_p(H)|^2}{p |z-z_H|^2} 
+ O ( 1 \s /\s[-2] p ) 
\end{align}
because $|\zv|^2 = |1_p|^2 = p$. Because $|\Lambda_p|^2 /p$ 
converges in $(0,\infty)$ as $p \upto \infty$ under our 
Assumption~\ref{asm:afm}\ref{afm:BB},
we achieve (in expectation) an asymptotically riskless portfolio
provided the convergence $|\scrE_p(H)|\to 0$ and irrespective
of its rate.

\section{Principal Component Analysis}
\label{sec:pca}

Let $\Y$ denote a $p \times n$ data matrix of $p$ variables
observed at $n$ dates which, for a random $n \times q$ matrix
$\X$ and random $p \times n$  matrix $\E$, follows the
linear model,
\begin{align} \label{data}
 \Y = B \X^\top \s[-1] +\s \E  \s .
\end{align}
The $p \times q$ matrix $B$ forms the unknown to be estimated,
and only $\Y$ is observed, while $\X$ is a matrix of 
latent variables and the matrix $\calE$ represents an additive noise.

The {\pca} estimate $H$ of $B$ may be derived from $q \ge 1$
leading terms of the spectral decomposition of the $p \times p$
sample covariance matrix $\Sam$ (see Remark \ref{rem:population}),
i.e., 
\begin{align} \label{spectral}
 \Sam 
= \tsum_{(\seig^2,\s h)} \seig^2 hh^\top = HH^\top + G
\end{align}
where the sum is over all eigenvalue/eigenvector pairs
$(\seig^2,h)$ for $h \in \bbR^p$ of unit length (i.e.,
$|h|=1$).  
The $j$th
column $\eta$ of the $p \times q$ matrix $H$ 
in $\req{spectral}$ is taken as $\eta =
\seig h$ where $\seig^2$ is the $j$th largest eigenvalue of
$\Sam$. The matrix $G = \Sam - HH^\top$ forms the residual.
Ordering the eigenvalues of $\Sam$ as $\seig^2_{1,p} \ge
\seig^2_{2,p} \ge \cdots \ge \seig^2_{p,p}$, we have 
\begin{align} \label{pcs}
  \scrH = H \qeig^{-1}_p \s; 
\s[32] H^\top H = \qeig_p^2  \s ,
\end{align}
where $\qeig_p^2$ is a $q \times q$ diagonal matrix with entries
$(\qeig_p^2)_{jj} = \seig^2_{j,p}$ and the columns
of the matrix $\scrH$ are the associated sample eigenvectors
$h$ in $\req{spectral}$ with $\scrH^\top \scrH
= \Id_q$.

 Since data is often centered in practice, in addition to
$\req{data}$, we consider the eigenvectors $\scrH$ of
the transformed $p \times n$ data
matrix $\Y \s \Jc$  where for any $\cn \in \bbR^n$, 
\begin{align} \label{Jc}
\s[64] \Jc = \Id - \frac{\s[5] \cn \s \cn^\top}{|\cn|^2} 
\s[32]
\big( \scrH = 
\nu_{p \times q}(\Y \Jc) = \nu_{p \times q}(\Sam)  \big) \s ,
\end{align}
and the sample covariance in $\req{spectral}$ is given by
$\Sam = \Y\s[1.5] \Jc \Y^\top\s[-1] / n$ since
$\Jc\s \Jc^\top = \Jc$.
Centering the $n$ columns of $\Y$ 
entails the choice $\cn = 1_n$ in $\req{Jc}$ but
we allow $\Jc = \Id$.

\begin{remark} \label{rem:population}
The identity $\Exp( \Sam) = \bSig = BB^\top  + \Sv$ 
is the aim of centering and holds under well-known conditions,
e.g., $\Y$ has i.i.d. columns,  $\Exp( \X^\top\s[-2] \Jc \X) \s / n = \Id$
with the $\X$ and $\E$  uncorrelated.
We do not require that $\Exp(\Sam) = \bSig$ for the
results of this section.
\end{remark}

 Our results require the following
signal-to-noise ratio (diagonal) matrix $\snr$, where the \tq{noise} is
specified in terms of the average of the bulk eigenvalues,  
$\kappa^2_p$ (c.f., $\req{bulk2}$).
\begin{align} \label{snr-bulk}
\s[64] \snr^2 =  \Id_q - \kappa^2_p \s  \qeig_p^{-2} \s; \s[32]
 \kappa^2_p = \frac{\sum_{j > q} \scrs^2_{j,p}}{n_+ -q}
\s[32] (n_+ > q) \s ,
\end{align}
where $n_+$ is the number of nonzero eigenvalues of $\Sam$. When
$p > n$, ensuring per $\req{snr-bulk}$ that $n_+ > q$  implies,
for $\Y$ of full rank, that $n_+ = n$ for $\Jc = \Id$ and $n_+ =
n-1$ otherwise.  For $p > n$, the eigenvectors $\nu_{n \times
q}(\Jc \Y^\top)$ and eigenvalues $\qeig_p^2$ may also be
computed more efficiently using the smaller $n \times n$ matrix
$\Jc \Y^\top \Y \Jc / n$ which shares its nonzero eigenvalues
with $\Sam$.  This computation is represented as follows
(c.f. $\req{Jc}$).
\begin{align} \label{efficient}
\s[32] \scrH
 = \Y \s \nu_{n \times q}(\Jc \Y^\top) \s \qeig_p^{-1}
 /\s[-1] \sqrt{n}
\end{align}

The {\pca}--estimated model for $\bSig = BB^\top + \Sv$ 
takes $H = \scrH \qeig_p$ in $\req{pcs}$ and  our
estimator $\hSig = HH^\top + \hat{\sv}^2 \Id$ for the simple choice
$\hat{\sv}^2 = n \kappa^2_p \s /\s[-1] p$ which suffices in view of
Section \ref{sec:qob}. We prove 
(Theorem \ref{thm:pca}) that $\hat{\sv}^2$ consistently
estimates the average idiosyncratic variance $\tr(\Sv)/ \s[-1]
p$ as $p \upto \infty$, under our upcoming Assumption
\ref{asm:asymp}.\footnote{The residual $G = \Sam - HH^\top$ is
typically regularized to form an robust estimate of $\Sv$.
Examples include zeroing out all but the diagonal of this
matrix,  and the POET estimator \cite{fan2013}.}

Sections \ref{sec:pca_optbias}--\ref{sec:pca_asymp} below define
$\scrB = \nu_{p \times q}(B)$, the $q$ eigenvectors of $BB^\top$
associated with the largest $q$ eigenvalues (as other choices
were possible for Theorem~\ref{thm:discrepancy} and the
definition of $\scrE_p(H)$ in $\req{optbias}$). We do not
require $\Exp(\Sam) = \bSig$ per Remark \ref{rem:population}.

\subsection{Norm of the optimization bias for {\pca}} 
\label{sec:pca_optbias}
We analyze the asymptotics $\scrE_p(H)$ for the {\pca} estimate
$H$.  Lemma \ref{lem:K} with $K = \qeig_p$ in $\req{pcs}$ and
$\scrH^\top \scrH = \Id$ imply that 
$z_H = HH^\dagger z = \scrH\scrH^\top z = z_\scrH$ 
and $\ip{z}{z_\scrH} = |z_\scrH|^2$, which reduces $\req{optbias}$ to
\begin{align} \label{optbias_gps}
\scrE_p(H)  = \scrE_p(\scrH) 
= \frac{\scrB^\top (z - z_\scrH)}
{ |z - z_\scrH| } 
= \frac{\scrB^\top z - (\scrB^\top \scrH)
\s (\scrH^\top z)}{\sqrt{1 - |\scrH^\top z|^2}} \s .
\end{align}

The unknowns in $\req{optbias_gps}$ are $\scrB^\top z$ and
$\scrB^\top \scrH$ and we provide theoretical evidence that they
cannot be estimated from data in Section \ref{sec:impossible}
without very strong assumptions. Here, we nevertheless obtain an
estimate of the length $|\scrE_p(H)|$.
The following addresses a division by zero in 
$\req{optbias_gps}$. Recall that
$z = \frac{\zv}{|\zv|}$ and
$\zv_B = BB^\dagger \zv$ per $\req{eH}$.  

\begin{assumption} \label{asm:eB}
 $\lsup |\zv_B|\s /\s |\zv| < 1$ 
and $\linp |\zv| \neq 0$ 
for $\zv = \zv_{p \times 1}$. 
\end{assumption}

Our next assumption concerns the
matrices $\X$ and $\E$ in $\req{data}$. These 
guarantee that almost all realizations of the data $\Y$
have full rank for sufficiently large $p$, allowing us
to treat $n_+$ in $\req{snr-bulk}$ as 
$n_+ = n$ when $\Jc = \Id$
and $n_+ = n-1$ otherwise.

\begin{assumption} \label{asm:asymp}
Assumption \ref{asm:afm} on the matrices $B = B_{p \times q}$
 and $\Sv = \Sv_{p \times p}$ holds and the following 
conditions hold 
for $\X$ and sequences 
$\E = \E_{p \times n}$ and $\zv = \zv_{p \times 1}$.
\begin{enumerate}[label=\text{(\alph*)}, itemsep=0.0in]
\item Only \textup{$\Y$} is observed 
(the variables \textup{$\X,\calE$} in $\req{data}$ are latent). 
\label{asm:latent}
\item The true number of factors $q$ is known and
$n_+ > q$ (with $n$ fixed). \label{asm:qnum}
\item \textup{$\X^\top \s[-2] \Jc \X$} is ($q \times q$) invertible 
almost surely (and does not depend on $p$).  \label{asm:XX}
\item $\limp\s \E^\top \E /p
= \sv^2 \Id$ almost surely for some constant $\sv > 0$.
\label{asm:EE}
\item $\lsup \|\s \Jc \s \E^\top B  \| \s / \s[-1] p  = 0$ 
almost surely for some
matrix norm $\|\s \cdot \s \|$ on $\bbR^{n \times q}$. \label{asm:EB}
\item $\lsup | \s \Jc\s \E^\top z|\s /\s[-1] \sqrt{p} = 0$
almost surely where $z = \zv / |\zv|$.  \label{asm:Ee}
\end{enumerate}
\end{assumption}

These conditions are discussed below.  Our fundamental result on
{\pca} (in conjunction with Theorem \ref{thm:discrepancy}) may
now be stated. Its proof is deferred to Appendix \ref{app:pca}.

\begin{theorem} \label{thm:pcabias}
Suppose Assumptions \ref{asm:eB} \& \ref{asm:asymp} hold. Then, 
almost surely, 
\begin{align} \label{pcabias}
\limp \Big( |\scrE_p(\scrH)|
 - \frac{ |\Pd \scrH^\top z| } {|z - z_\scrH|} \Big) = 0 \s ,
\end{align}
where $\Pd =  (\snr^{-1} -\snr)$. Moreover, 
the length of the \textup{\pca} optimization bias
$|\scrE_p(\scrH)|$ is eventually in $(0,\infty)$ almost surely, 
provided $\limp \scrH^\top z \neq 0_q$
(see Corollary \ref{cor:Bzero}).
\end{theorem}

We remark that  $\hz = \frac{\Pd \scrH^\top z } {|z - z_\scrH|}
\in \bbR^q$ is computable solely from the data  $\Y$ with almost
every $|\hz|$ bounded in $[0,\infty)$ eventually.  Theorem
\ref{thm:pcabias} demonstrates that {\pca}, and sample
eigenvectors $\scrH$ in particular, lead to poor \tq{plug-in}
covariance estimators for  quadratic optimization unless every
column of $\scrH$ is eventually orthogonal to $\zv$.  So
typically, the discrepancy $\mf_p$ in $\req{QhxD}$ between the
estimated and realized optima diverges to $-\infty$ as $p$ grows
and at a linear rate. In the portfolio application of
Section~\ref{sec:risk}, this covariance results in
strictly positive expected (out-of-sample) portfolio risk
$\tru_p$ per $\req{tru}$--$\req{tru_asymp}$ asymptotically,
which may be approximated by using $|\hz|$.

We make some remarks on Assumption \ref{asm:asymp}. Conditions
\ref{asm:latent}--\ref{asm:XX} are straightforward, but we
mention that the invertibility of $\X^\top\s[-2] \Jc \X$ is 
closely related to the requirement that $n_+ > q$ in condition
\ref{asm:qnum}.
Condition~\ref{asm:XX}
fails when $\cn$ in $\req{Jc}$ lies in $\col{\X}$ but such a case is dealt
with by rewriting the data in $\req{data}$ as  
$\Y = \alpha \cn^\top + B_\alpha
\X^\top_\alpha + \E$ for some $B_\alpha$ and $\X_\alpha$ of $q-1$
columns each, and some mean vector $\alpha \in \bbR^p$.  Then, 
we have $\Y \Jc = B_\alpha
\X^\top_\alpha \Jc + \E \Jc$, and it
only remains to check if condition \ref{asm:XX} holds with
the matrix $\X_\alpha$ replacing $\X$.
Conditions \ref{asm:EE}--\ref{asm:Ee} require that
strong laws of large numbers hold for 
the columns  of the sequence $\calE =
\calE_{p \times n}$. These roughly state that the columns of
$\calE$ are stationary with weakly dependent entries having
bounded fourth moments. All three are easily verified for
the $\calE_{p \times n}$ populated by i.i.d. Gaussian random
entries. Lastly, we remark that if 
conditions \ref{asm:EB} and \ref{asm:Ee} hold for
$\Jc = \Id$ they hold for any $\Jc$.  

Since in practice both $n$ and $p$ are finite, we can make some
refinements to the definitions in $\req{snr-bulk}$ based
on some classical random matrix theory. In particular, it is
well known that when the aspect ratio $\np$ converges
in $[0,\infty)$ (in our
case to zero), the eigenvalues of $\E^\top \E/p$
have support that is approximately between
$\sv^2 (1 - \sqrt{\np})^2$ and $\sv^2 (1 + \sqrt{\np})^2$  
for the constant $\sv^2$ in condition \ref{asm:EE}.
We can then define,
\begin{align} \label{bulk2}
  \kappa_p^2  = 
\bigg( \frac{1 + \np}{n_+ - q + \np} \bigg)
\sum_{j > q} \scrs^2_{j,p}
\end{align}
which is a Marchenko-Pastur type adjustment
to $\kappa_p^2$ (and $\snr^2$) defined in $\req{snr-bulk}$.
When the eigenvalues of $\E^\top \E/p$ obey the Marchenko-Pastur
law, this $\kappa^2_p$ is advisable.

\subsection{{\hl} results for {\pca}} 
\label{sec:pca_asymp}
Theorem \ref{thm:pcabias} is essentially a
corollary of our next result, which is of independent
theoretical interest for the {\hd} literature.

\begin{theorem} \label{thm:pca}
Suppose Assumption \ref{asm:asymp} holds. Then, hold almost
surely.
\begin{enumerate}[label=\text{(\alph*)}, itemsep=0.04in]
\item $\limp  \Psi \qeig_p K_p^{-1} = \Id$ where $K_p$ is a 
$q \times q$ diagonal matrix with $(K^2_p)_{jj}$, the $j$th 
largest eigenvalue of the $p \times p$ matrix  
$B\s V_n B^\top$ where
$V_n = \X^\top\Jc \X / n$.
\label{pca:Sp}
\item $ \limp  \frac{n \kappa^2_p}{p \sv^2}  = 1$ for
the constant $\sv$ of Assumption \ref{asm:asymp}\ref{asm:EE}. 
\label{pca:kappa}
\item $\limp |\scrH^\top \scrB \scrB^\top \scrH - \snr^2 |
= 0$ and every $\Psi^2_{\jj}$ is eventually in $(0,1)$.
\label{pca:HBBH}
\item $\limp |\scrH^\top z -(\scrH^\top \scrB) \scrB^\top z| = 0$.
\label{pca:HBBz}
\end{enumerate}
\end{theorem}

The proof is deferred to Appendix \ref{app:pca}. Parts
\ref{pca:Sp}--\ref{pca:kappa} should not surprise those
well versed in the {\hl} literature. Nevertheless, these
limit theorems for eigenvalues provide new content
by supplying estimators, not just asymptotic descriptions.

\begin{remark} \label{rem:eigenvals}
Parts \ref{pca:Sp}--\ref{pca:kappa} of Theorem \ref{thm:pca}
supply improved eigenvalue estimates for the \textup{\pca}
covariance model when $\Exp(\Sam) = \bSig$, and while these have
no effect on the optimization bias $\scrE_p(\scrH)$, we
summarize them.  Part \ref{pca:Sp} implies $H \Psi^2 H^\top =
\scrH (\qeig_p \snr)^2 \scrH^\top$ is an improved estimator
(relative to $HH^\top$) of the population matrix $BB^\top$. Part
\ref{pca:kappa} implies that  $\hat{\sv}^2 = n \kappa_p^2\s
/\s[-1] p$ is an asymptotic estimator of $\tr(\Sv)/p$ where $\Sv
= \Exp (\calE \Jc \calE^\top\s[-2]/n)$.
To see this, w.l.o.g. take $\Jc = \Id$, and note that the trace
$\tr$ and the expectation $\Exp$ commute. Then, $\tr\s (\Sv) =
\tr\s (\Exp (\calE \calE^\top/n)) =   (p/n)\s \Exp (\tr \s
(\calE^\top \calE / \s[-1] p)) $ and since $\calE^\top \calE
/\s[-1] p \to \sv^2 \Id_{n \times n}$
(Assumption~\ref{asm:asymp}\ref{asm:EE}) provided $\calE^\top
\calE/\s[-1] p$ is uniformly integrable, $\tr (\Sv)/\s[-1]p$
converges to $\sv^2$.
\end{remark}

The limits in parts \ref{pca:HBBH}--\ref{pca:HBBz} 
of Theorem \ref{thm:pca} are new and
noteworthy.  They supply estimators for the quantities
$\scrH^\top \scrB \scrB^\top \scrH$ and $\scrH^\top z_B =
(\scrH^\top \scrB) \scrB^\top z$ from data. While these
are not enough to estimate $\scrE_p(\scrH)$ (for that
we need both $\scrB^\top \scrH$ and $\scrB^\top z$), they
suffice for the task of estimating the norm $|\scrE_p(\scrH)|$
from the data $\Y$.

The convergence in part \ref{pca:HBBH} has an interpretation. By
direct calculation we have that $\scrH^\top \scrB\scrB^\top\scrH =
(BB^\dagger \scrH )^\top (BB^\dagger \scrH)$ (e.g., see
$\req{HBBH}$ in Appendix \ref{app:pca}), which implies that for
columns $j, j'$ of $\scrH$, say $h$ and $h'$, the $\jj'$th
entry of $\scrH^\top \scrB \scrB^\top \scrH$ is
$\ip{h_B}{h'_B}$, i.e., the inner product of $h$ and $h'$
projected onto $\col{B}$. This is in contrast to the $\jj'$th
entry $\ip{h}{b'}$ of $\scrH^\top \scrB$ where $b'$ is the $j'$th
column of $\scrB$. Part \ref{pca:HBBH} states
that, 
\begin{equation} \label{hBh_B}
\begin{aligned} 
\s[32] \limp \s \ip{h_B}{h'_B} &= 0  \s[32] (h \neq h') \\ 
 \limp  |h_B| \s \Psi^{-1}_{\jj} &= 1  \s[32] (j = j')
\end{aligned}
\end{equation}
almost surely, where $\Psi_{\jj}$ is itself a random sequence
eventually in $(0,1)$.  That is, sample eigenvectors remain
orthogonal in $\col{B}$, but their norms are less than the
maximal unit length, i.e., columns of $\scrH$ are inconsistent
estimators of columns of $\scrB$.

The following elegant characterization is an artifact of the
fact that square matrices with orthonormal rows must also have
orthonormal columns. 

\begin{corollary}  \label{cor:inv}
Let $\calH = \scrH \Psi^{-1}$. Under the hypotheses of 
Theorem \ref{thm:pca} the $q \times q$ matrices $\calH^\top
\scrB$ and $\scrB^\top \calH$ are asymptotic inverses
of one another., i.e., almost surely,
\begin{align} \label{unitary}  
\limp |\s \calH^\top  \scrB \scrB^\top \calH  - \Id \s | = 
\limp |\s \scrB^\top \calH \calH^\top  \scrB - \Id \s | = 0  \s .
\end{align}
\end{corollary}

Applying this to 
$\limp |\calH^\top z - \calH^\top \scrB \scrB^\top z| = 0$ 
per Theorem \ref{thm:pca}\ref{pca:HBBz}, yields
\begin{align} \label{Hzlen}  
\limp \frac{|z_B|}{|\calH^\top z|}  = 1
\end{align}
provisionally on Assumption \ref{asm:eB} and without it, both
$|z_B| = |\scrB^\top z|$ and $|\calH^\top z|$ converge to zero.
Thus, Theorem \ref{thm:pca} implies we can asymptotically
know the length $|z_B|$ (the norm of the projection of $z$
in $\col{B}$). Further, as all diagonal entries of $\Psi$ are 
eventually smaller
than one and $|\calH^\top z| \le |\Psi^{-1}||\scrH^\top z|
= |\Psi|^{-1} |z_H|$, we deduce that $\col{B}$ has larger
projection onto $z$ than does $\col{H}$ eventually in $p$. We
conclude with a simple consequence
of Theorem \ref{thm:pca}\ref{pca:HBBz} relevant
for Theorem \ref{thm:pcabias}.

\begin{corollary} \label{cor:Bzero} Suppose that Assumption 
\ref{asm:asymp} holds. Then, $\limp \scrB^\top z
= 0_q$ implies $\limp \scrH^\top z = 0_q$ almost surely.
\end{corollary}

\section{An Impossibility Theorem}
\label{sec:impossible}

The problem of estimating the unknown $\scrB^\top \scrH$ and
$\scrB^\top z$ appearing in $\req{optbias_gps}$ encounters
significant challenges for $q > 1$. It is related to, but
separate from, the problem called \tq{\it unidentifiability}
that arises in the context of factor analysis (e.g.,
\ci{shapiro1985}).  Here, we prove an \tq{\it impossibility}
result.  To give an interpretation of $\scrB^\top \scrH$, we now
require $\Exp(\Sam) = \bSig$ and $\scrB = \nu_{p \times q}(B)$,
so that $\scrH$ and $\scrB$ may be regarded as the sample and
the population eigenvectors (or asymptotic principal
components).

With $\scrB \Lambda_p \scrW^\top$ denoting the  singular value
decomposition of $B \in \bbR^{p \times q}$ in $\req{data}$, and 
similarly $ (1/ \s[-2] \sqrt{n})\s \Y \scrU  = \scrH \qeig_p $
with $\scrU \in \bbR^{n \times q}$, we find (see Appendix
\ref{app:pca}),
\begin{equation} \label{BHlim_text}
  \limp  \scrB^\top \scrH 
= \limp \Lambda_p \scrW^\top \X^\top \scrU \scrS_p^{-1}  /\sqrt{n}
\end{equation} 
which holds almost surely under Assumption~\ref{asm:asymp}. This
limit relation has been studied in the {\hd} literature under
various conditions and modes of convergence (e.g., \ci{jung2012}
and \ci{shen2016}). But these authors do not derive estimators
for the right side of $\req{BHlim_text}$ (i.e., the $\Lambda_p
\scrW^\top \X^\top$ is not observed).

We prove that it is not possible, without very strong assumption
on $\X$, to develop asymptotic estimators of the inner product
matrix $\scrB^\top \scrH$. Given this, it is also reasonable to
conjecture the same for $\scrE_p(\scrH) \in \bbR^q$.  While this
problem is motivated by our study of the quadratic optimization
bias, the estimation of the entries of $\scrB^\top \scrH$, and
hence the estimation of angles between the sample and population
eigenvectors, is an interesting (and to our knowledge,
uninvestigated) problem in its own right.  

We remark that the problem of \tq{unidentifiability} amounts to
the observation that replacing $B$ and $\X$ by  $B\scrO $ and
$\X \scrO$ for any orthogonal matrix $\scrO$ does not alter the
observed data matrix $Y = B\X^{\top}+ \E$ deeming $B$
unidentifiable (i.e., $B$ or $B\scrO$?).  However, the quantity
of interest in our work is $\scrB^{\top}\scrH$, which bypasses
this type of unidentifiability as $\scrB$  is defined via the
identity $\nu_{p\times q}(B)=\nu_{p\times q}(B\scrO)$ which is a
population quantity encoding the uniquely selected $q$
eigenvectors of $\Exp(\Sam) = \bSig$. Hence, the
unidentifability of $B$ is related to but not the same as
problem we formulate.

We work in a setting where the noise $\E$ in $\req{data}$ is
null and the matrices $B = B_{p \times q}$ have additional
regularity over Assumption \ref{asm:afm}.  The presumption here
is that these simplifications can make our stated estimation
problem for $\scrB^\top \scrH$ only easier.

\begin{condition} \label{cond:simple}
The data matrices $\Y = \Y_{p \times n}$ with $n > q$ fixed 
have $\Y = B \X$ for a sequence $B = B_{p \times q}$ and 
$\X \in \bbR^{n \times q}$ satisfying the following.
\begin{enumerate}[label=\text{(\alph*)}, itemsep=0.0in]
\item  \label{simple:X} The $\X$ is a random variable on
a probability space $(\Omega, \calF, \Prb)$ with $V^2_n = 
\X^\top \X/n$  almost surely invertible and such that 
$\Exp(V^2_n) = \Id_q$.
\item \label{simple:B} The $B = B_{p \times q}$ (for all $p$) satisfies
$B /\s[-2] \sqrt{p} = \scrB \Lambda \scrW^\top$ 
for $\scrB = \nu_{p \times q} (B)$,
 a fixed $q \times q$ orthogonal $\scrW$ and fixed $q \times q$
diagonal $\Lambda$ with $\Lambda_{ii} \neq \Lambda_{jj}$ for 
all $i \neq j$.
\end{enumerate}
\end{condition}

Any $(B, \X)$ of Condition \ref{cond:simple} has $B =
B_{p \times q}$ obeying Assumption \ref{asm:afm}\ref{afm:BB}
with $\limp B^\top B/p = \scrW \Lambda^2 \scrW^\top$, and $\X$ for which
the sample covariance $\Sam = \Y \Y^\top\s[-2] /n$ satisfies
$\Exp(\Sam) = \bSig = BB^\top = \scrB \Lambda_p^2 \scrB^\top$
for the eigenvalue matrix $\Lambda_p^2 = p \Lambda^2$.

For $B = B_{p \times q}$ satisfying Condition
\ref{cond:simple}\ref{simple:B}, we define a set of orthogonal
transformations which non-trivially change the eigenvectors 
$\scrB = \nu_{p \times q}(BB^\top)$. Let,
\begin{align} \label{Ond}
 \Ond =  \big\{ \scrO \in \bbR^{q \times q}
\cst    \nu_{p \times q}(\scrB_o \Lambda \scrW^\top) = \scrB_o
= \scrB \scrO \text{ for all $p$} \s  \big\} \s . 
\end{align}

Every element $\scrO \in \Ond$ induces the data $\Y = B_o
\X^\top$ with $B_o = \scrB_o \Lambda \scrW^\top \sqrt{p}$ and
$\scrB_o = \scrB \scrO$, which is uniquely identified by the
orthogonal matrix $\scrO$. The new data set built in this way
satisfies Condition \ref{cond:simple} for $(B, \X)$
of that condition. We remark that the only diagonal
element of $\Ond$ is the identity matrix $\Id_q$ (i.e., flipping
the signs of any of the columns of $\scrB$ does not result in a
different set of eigenvectors $\nu_{p \times q}(B)$).  Indeed,
if we partition all $q \times q$ orthogonal matrices by the
equivalence relation that sets two matrices equivalent when
their columns differ only by a sign, then the set $\Ond$ selects
exactly one element from each equivalence class. Since the
number of elements in each equivalence class is finite, we have
established that the set $\Ond$ has the same cardinality as the
set of all orthogonal matrices with dimensions $q\times q$.

We now consider $\bbG \subseteq \Ond$  and a sequence of
(nonrandom) measurable functions $\f_p \cst \bbR^{p \times q}
\to \bbR^{q \times q}$ that together with the notation $\Y_{B} =
B \X^\top$ define,
\begin{align} \label{AfB} 
A_{\f,\bbG}(\scrB) =
\{ \omega \in \Omega \cst \limp | \scrH^\top \scrB_o -
f_p(\Y_{B_o})|(\omega) = 0, 
\s[2] \scrO \in \bbG \cup \{\Id_q \} \}\s  .
\end{align}

The event $A_{f,\bbG}(\scrB)$ consists of all outcomes
for which  the $f_p(\Y_{B_o})$
consistently estimate $\scrH^\top \scrB_o $ as $p
\upto \infty$ for every $\scrO \in \bbG\cup \{I_q\}$. 
The following lemma may be used to generate bounds on the
probability of event $A_{\f,\bbG}(\scrB)$ for many examples. 

\begin{lemma} \label{lem:gf}
Suppose Condition \ref{cond:simple}. Then, for any $\f 
\cst \bbR^{p \times n} \to \bbR^{p \times q}$
and corresponding function $f^\nu$ given by
$\f^\nu(\Y/\s[-1] \sqrt{p}) 
= \nu_{p \times q} (\Y)\f(\Y) = \scrH\f(\Y)$, we have
\begin{align}
| \scrB - \s \f^\nu(\Y/\s[-1] \sqrt{p})| \le
| \scrH^\top \scrB - \f(\Y)|  \s .
\end{align}
\end{lemma}

\begin{proof}
Since almost surely, $\X$ has linearly independent columns,
it is easy to see that $\scrB = \scrH\scrH^\top \scrB$.
Using that $|\scrH| =1$ and that $\f^\nu(\Y) = \scrH \f(\Y)$
 yields,
 \begin{align*} 
   |\scrB - \f^\nu(\Y/\s[-1]\sqrt{p})| 
&= |\scrB - \scrH \f(\Y)|  
= |\scrH\scrH^{\top}\scrB - \scrH \f(\Y)| 
\\&\leq |\scrH||\scrH^{\top}\scrB - \f(\Y)| 
= |\scrH^{\top}\scrB - \f(\Y)| 
\end{align*}
\end{proof}

\vspace{-0.16in}
The next example is a good warm-up for our main 
result (Theorem \ref{thm:angles})
below.

\begin{example} 
Let $m \in \bbR^{n \times q}$ be nonrandom and  $\Xv =
\X \scrW \Lambda$, a random matrix with $\varphi_m = \Prb( \Xv = m)$
and $\Prb( \Xv = m \scrO) = 1 - \varphi_m$ for some $\scrO \in
\Ond \setminus \{I_q\}$.  By taking $\bbG = \{I_q, \scrO\}$, the
event $A_{\f,\bbG}(\scrB)$ contains the outcomes for which
$\scrH^\top \scrB_o$ admits a consistent estimator for two
data sets corresponding to the $\scrB$ and $\scrB_o = \scrB \scrO$. 

If $\Prb(A_{\f,\bbG}(\scrB)) > \varphi_m$, then
$A_{\f,\bbG}(\scrB)$  contains outcomes corresponding to
each possible realization of $\Xv$ which implies by Lemma
\ref{lem:gf} that both $|\scrB - \f^\nu_p (\scrB m^\top)|$ and
$|\scrB_o - \f^\nu_p(\scrB m^\top)|$ converge to zero.
Since this is a contradiction, 
$\Prb(A_{\f,\bbG}(\scrB)) \le \varphi_m$. 
\end{example}

This stylized example may be substantially generalized by requiring
a certain distributional property of the random 
variable $\Xv = \X \scrW \Lambda$.

\begin{definition}  \label{Gdist}
We say a random variable $\Xv \in \bbR^{n \times q}$
is  \emph{$\bbG$-distributable} if there exists 
a collection $\bbG \subseteq \Ond \setminus \{\Id_q\}$  
such that
for any measurable $G \subseteq \bbR^{n \times q}$,
\[  \s[32] \Prb( \Xv \in G) \s \le \s \Prb(\Xv \in \cup_{\scrO
\in \bbG}  G\scrO ) \s[16] 
\big( G \scrO = \{ m \scrO \cst m \in G\} \big) \s . 
\]
\end{definition}

Clearly, $\Xv$ that has mean zero i.i.d. Gaussian entries is
$\bbG$-distributable for $\bbG$ with just one element, but we expect many 
random matrices $\Xv$ to have this property.
Our main result shows that even when restricting to a smaller
set of covariance models, the chances of estimating the 
matrix $\scrH^\top \scrB$ are no better
than a coin flip.

\begin{theorem} \label{thm:angles}
Suppose  Condition \ref{cond:simple} holds and $\Xv = \X \scrW
\Lambda$ is 
$\bbG$-distributable with $\bbG \subseteq \Ond \setminus
\{\Id_q\}$. 
Then, for this $\bbG$ and
any sequence of (nonrandom) measurable  functions $\f_p 
\cst \bbR^{p \times n} \to \bbR^{p \times q}$,
the  $A_{\f,\bbG}(\scrB)$ in $\req{AfB}$ 
has $\Prb(A_{\f,\bbG}(\scrB)) \le 1/2$.
\end{theorem}

\begin{proof}  By Lemma \ref{lem:gf}, we have 
$A_{\f,\bbG}(\scrB) \subseteq A^\nu_{\f,\bbG}(\scrB)$ where
\[ A^\nu_{\f, \bbG}(\scrB) =  \{
\omega \in \Omega \cst \limp |  \scrB_o -
 f^\nu_p(\scrB_o \Xv^\top)|(\omega) = 0, \s[2] 
\s \forall \scrO \in \bbG\cup \{I_q\}\} \]
for $\scrB_o \Xv^\top \s[-2] \sqrt{p} = \Y_{B_o}  = \Y$, the data 
matrix per $\req{AfB}$, $\scrB_o = \scrB \scrO$
and $\Xv = \X \scrW \Lambda$, after recalling the definition
$\f^\nu(\Y/\s[-1] \sqrt{p}) 
= \nu_{p \times q} (\Y)\f(\Y) = \scrH \f(\Y)$.

Note that the $\calF$-measurability  of 
the set $A^\nu_{\f,\bbG}(\scrB)$ is
granted by the measurability of each $\f^\nu_p$ (i.e., each
$\f_p$ is measurable and so is each $\nu_{p \times q}$
\citep{acker1974}).

Letting $G = \Xv(A^\nu_{\f,\bbG}(\scrB)) = \{ \Xv(\omega) \in 
\bbR^{n \times q} \cst \omega \in A_{\f,\bbG}^\nu(\scrB)\}$, we see that
\begin{align} \label{fvB}
\s[32] \limp | \scrB  - \f^\nu_p(\scrB m^\top )| &= 0
\s[32] \forall m \in G  ,  
\end{align}
by taking $\scrO = \Id_q$. Analogously, for $\scrB_o
= \scrB \scrO$ for any $\scrO \in \bbG$, we have
\begin{align} \label{fvBO}
\s[10] \limp | \scrB\scrO  - \f^\nu_p(\scrB \scrO m^\top )| &= 0
\s[32] \forall m \in G   .
\end{align}

Letting $G' = \underset{\scrO \in \bbG}{\cup} G \scrO$, we claim
that $G$ and $G'$ are disjoint. To see this, note that if 
$m_1\in G\cap G'$,
then $m_1 =  m_2\scrO$ for $m_2\in G$, $\scrO \in \bbG$.
Substituting $m_1 = m_2\scrO \in G$ for $m$ in relation
$\req{fvBO}$, and substituting $m_2 \in G$ for  $m$ in relation
$\req{fvB}$, yields
\begin{align*}
\limp | \scrB\scrO  - \f^\nu_p(\scrB  m_2^\top )| = 0
\s[4] \text{ and } \s[4]
\limp | \scrB - \f^\nu_p(\scrB  m_2^\top )| = 0 \s ,
\end{align*}
a contradiction, as both cannot hold simultaneously.
Thus, $G$ and $G'$ are disjoint.

Consequently $\{\Xv\in G\}$ and
$\{\Xv\in G'\}$ are disjoint and  moreover, the $\bbG$-distributability
of $\Xv$ implies $\Prb(\Xv \in G) \le \Prb(\Xv \in G')$.
This along with the fact that $A_{\f,\bbG}(\scrB)\subseteq
A_{\f,\bbG}^{\nu}(\scrB)\subseteq \{\Xv\in G\}$ implies 
the desired result, i.e., 
\begin{align*}
    1 &\ge \Prb(M\in G') + \Prb(M\in G) 
    \ge 2 \s \Prb(M\in G) 
    \ge 2\s \Prb(A^\nu_{\f,\bbG}(\scrB)) 
    \\ &\ge 2\s \Prb(A_{\f,\bbG}(\scrB)) \s .
\end{align*}
\end{proof}

\section{Optimization Bias Free Covariance Estimator}
\label{sec:Hsharp}

Let $\scrH$ be the $p \times q$ matrix of eigenvectors in 
$\req{efficient}$ of the sample covariance $\Sam$.
Recalling the variables $\Pd = (\snr^{-1} - \snr)$
and $z \in \bbR^p$ of
Theorem \ref{thm:pcabias}, we define
\begin{align} \label{zpH} 
 z_{\perp \scrH} 
= \frac{z - z_\scrH}{|z-z_\scrH|}  \in \bbR^p \s, 
\s[16]
\hz= \frac{\Pd \scrH^\top z} 
{|z - z_\scrH|}  \in \bbR^q \s ,
\end{align}
Theorem \ref{thm:pcabias} proved that 
$\limp (|\scrE_p(\scrH)| - |\varphi|) =0$ with $|\hz|$ eventually
in $[0,\infty)$ almost surely.
From the observable $\hz$ and $\zpH$, we 
now construct an $p \times q$ matrix $\scrH_\sharp$ with
$\scrE_p(\scrH_\sharp) \to 0$ as $p \upto \infty$.
To this end, consider the eigenvalue decomposition 
\begin{align} \label{Phi}
 \snr^2 + \hz \hz^\top
= \Rv \Phi^2 \Rv^\top\s ,
\end{align}
for eigenvectors $\Rv = \nu_{q \times q}(\snr^2 + \hz\hz^\top)$
and diagonal $q \times q$ matrix of eigenvalues $\Phi^2$.
The estimator $\scrH_\sharp$ is computed 
as the eigenvectors $\scrH_\sharp =
\nu_{p \times q}(\scrH \Psi + \zpH \hz^\top)$, i.e.,
\begin{align} \label{Hsharp}
 \scrH_\sharp = (\scrH \Psi + \zpH \s \hz^\top)  \Rv
\Phi^{-1} ,
\end{align}
where the diagonal $\Phi$ is invertible for $p$ sufficiently large
under our assumptions.

\begin{theorem}\label{thm:Hsharp}
Suppose Assumptions \ref{asm:eB} \& \ref{asm:asymp} hold. Then, 
almost surely, 
\begin{align}  \label{limEHsharp}
\limp \scrE_p(\scrH_\sharp) = 0_q \s . 
\end{align}
Moreover, $\scrH_\sharp^\top \scrH_\sharp = \Id$
and $\limp | \scrH_\sharp^\top \scrB\scrB^\top \scrH_\sharp
-  \Phi^2|  = 0$ almost surely.
\end{theorem}

The proof is deferred to Appendix \ref{app:sharp} but we
sketch the derivation of 
$\req{Hsharp}$ and give a geometrical interpretation 
in Sections \ref{sec:gps} and \ref{sec:sketch}.
We take $H_\sharp = \scrH_\sharp \Psi \qeig_p$ to combine the
eigenvector correction $\scrH_\sharp$ with that for the
eigenvalues (see Remark \ref{rem:eigenvals}). Note that
$\scrE_p(\scrH_\sharp) = \scrE_p(H_\sharp)$ by Lemma
\ref{lem:K}. We let
$\hSig_\sharp = H_\sharp H_\sharp^\top + \hat{\sv}^2 \Id$ 
be our covariance estimator where,
identically to {\pca},
we take
$\hat{\sv}^2 = n \kappa_p^2\s /\s[-1] p$ with $\kappa_p^2$
in $\req{snr-bulk}$ or $\req{bulk2}$.

Theorem
\ref{thm:Hsharp} provides theoretical guarantees for many
applications, including that the estimator
$\hSig_\sharp$ is now demonstrated to yield minimum variance
portfolios (i.e., solutions of $\req{minvar}$) with zero
asymptotic variance  (see $\tru^2_p$ in $\req{tru_asymp}$).
Addressing the convergence rate of $\req{limEHsharp}$ is outside
of our scope, but we study this rate numerically in
Section~\ref{sec:numerics}, which shows, at least for Gaussian data,
that rate is $O(1/\s[-1] \sqrt{p})$. This suggests that
$\scrH_\sharp$ yields a bounded discrepancy $\mf_p$ of Theorem
\ref{thm:discrepancy} under some conditions. 

The last part of Theorem \ref{thm:Hsharp} concerns the inner
products of the columns of $\scrH_\sharp$ projected
onto $\col{B}$. This is in direct comparison to
Theorem \ref{thm:pca}\ref{pca:HBBH} which shows that the sample
eigenvectors $\scrH$ are orthogonal in $\col{B}$ and the same is 
true for the columns of $\scrH_\sharp$ since $\Phi^2$ is diagonal.
Selecting the $j$th column $h_\sharp$ of $\scrH_\sharp$ we have,
\[  \limp   | \s h_{\sharp B}| \s  \Phi^{-1}_{\jj} = 1 \s , \]
as compared with 
$\limp   | \s h_{B}| \s  \snr^{-1}_{\jj} = 1$
in
$\req{hBh_B}$. Note,
$\snr^2_{\jj} \le \Phi^2_{\jj}$ eventually with a strict inequality
when $|\varphi|$ is bounded away from zero (in $p$) due to
$\req{Phi}$. Thus the length $|h_{\sharp B}|$ of 
$h_\sharp$ projected onto $\col{B}$ is at least as large
as for its counterpart 

\subsection{Remarks on the GPS program}
\label{sec:gps}
The special case $q = 1$ was considered by \ci{goldberg2022} (henceforth
{\gps}) who apply their results to portfolio theory. We summarize
the relevant parts of the {\gps} program making adjustments for
greater generality and compatibility with our solution in
Section \ref{sec:sketch}.

Here, $\req{bSig}$ takes the form
$\bSig = \beta \beta^\top + \Sv$ where $\beta \in \bbR^p$
and Assumption \ref{asm:afm} requiring  a
sequence $\beta = \beta_{p \times 1}$ for which
$\ip{\beta}{\beta}/p$ converges in $(0, \infty)$. The
sample covariance matrix may be written as
$\Sam = \eta \eta^\top + G$ where $\seig^2 = 
\ip{\eta}{\eta}
= \max_{v \in \bbR^p} \ip{v}{\Sam v} / \ip{v}{v}$
is the largest eigenvalue with eigenvector $h = \eta / |\eta|$
and the matrix $G$ contains the remaining spectrum per 
$\req{spectral}$. Setting $b = \beta / |\beta|$ yields,
\begin{align} \label{Eph}
\scrE_p(h)  
&= \frac{\ip{b}{z} -\ip{b}{h} \ip{h}{z}}{\sqrt{1 - \ip{h}{z}^2}}
\end{align}
for the quadratic
optimization bias $\req{optbias_gps}$ in the case 
$q = 1$. Our $\req{Eph}$ uses a different
denominator than  GPS, but this difference is	not essential.
Our $\Sv$ generalizes the choice of a scalar
matrix in GPS and our Assumption \ref{asm:asymp} relax
their conditions.

The GPS program assumes (w.l.o.g.) that $\ip{b}{z} \ge 0$ and 
$\ip{h}{z} \ge 0$, enforces Assumption \ref{asm:eB}
so that $\lsup \ip{b}{z} < 1$, and takes the following steps.

\begin{enumerate}[label=\text{(\arabic*)},itemsep=0.0in,leftmargin=*]
\item \label{gps1}
Find asymptotic estimators for
 unknowns $\ip{b}{z}$ and $\ip{h}{b}$ in $\req{Eph}$. To this end, 
for the observed $\psi^2 =
1 - \frac{\s \tr(G)/\seig^2}{(n_+-1)}$ (c.f., $\req{snr-bulk}$), 
under Assumption \ref{asm:asymp} 
almost surely,
\begin{align} \label{gps}
 \limp  | \ip{b}{h} - \psi| = 0 < \limp \psi < 1 
\text{ and }
\limp  | \ip{h}{z} - \ip{h}{b} \ip{b}{z} | = 0  \s .
\end{align}
\item \label{gps2} Consider the estimator
$h_z t= \frac{h + t z}{|h + t z|}$ parametrized by $t \in \bbR$
so that $\ip{h_z t}{z}$ increases in $t \ge 0$. This 
construction is motivated by the $\ip{h}{b}\ip{h}{z}$ in the 
numerator of $\req{Eph}$ 
becoming eventually less than $\ip{b}{z}$ almost surely, per $\req{gps}$.

\item \label{gps3} Solve $\scrE_p(h_zt) = 0$ for $t = \tau_*$ as 
a function of the unknowns $\ip{h}{b}$ and $\ip{b}{z}$.
Leveraging $\req{gps}$, construct an observable $\tau$ 
such that $|\tau_* - \tau| \to 0$ as $p \upto \infty$ and
prove a uniform continuity of $\scrE_p(\s \cdot \s )$
to establish  that $\limp \scrE_p(h_z \tau) = 0$.
\end{enumerate}

These steps cannot be easily extended to the setting of general
$q$.  Step \ref{gps1} is no longer possible in view of Theorem
\ref{thm:angles}, and indeed, the \tq{sign} conventions
$\ip{b}{z} \ge 0$ and $\ip{h}{z} \ge 0$ cannot be appropriated
from the univariate case given that result. Step \ref{gps2} is
difficult to extend because its intuition becomes obscure for
general $\bbR^q$ where the vector $\scrE_p(\scrH)$ resides. Step
\ref{gps3} relies on basic calculations to determine the root of
a univariate function.  Determining roots in $\bbR^q$,
especially without the right parametrization in step \ref{gps2},
appears difficult given the definition of $\scrE_p(\s \cdot \s)$
in $\req{optbias_gps}$.


We make some adjustments to prime our
approach in Section \ref{sec:sketch}. First, write
\begin{align} \label{gpsq}
 \limp  | \ip{b}{h}^2 - \psi^2| =
\limp  | \ip{h}{z} - \ip{h}{b} \ip{b}{z} | = 0  \s .
\end{align}
as a replacement for $\req{gps}$.
This drops the sign conventions on $\ip{b}{z}, \ip{h}{z}$
to reformulate step \ref{gps1} for compatibility with 
the findings of Theorem \ref{thm:pca} parts \ref{pca:HBBH}--\ref{pca:HBBz}.

Our adjustment to step \ref{gps2} sacrifices its intuition
for additional degrees of freedom. In particular,
for $z_{\perp h} = \frac{z - z_h}{|z - z_h|}$
and $z_h = \ip{h}{z} h$ (c.f., $\req{zpH}$), set
\begin{align} \label{hzt}
\s[32] h_z t =  t_1 \s h +  t_2 \s z_{\perp h}
\s[32] (t_1, t_2 \in \bbR \cst  t_1^2 + t^2_2 =1) \s . 
\end{align}
 This two-parameter estimator $h_z t$ parametrizes the 
quadratic optimization bias as,
\begin{align} \label{Ehzt}
\scrE_p(h_z \s t)  
&= \frac{\ip{b}{z} -
(t_1 \ip{b}{h} + t_2 \scrE_p(h) ) 
(t_1 \ip{h}{z} + t_2 \sqrt{1 - \ip{h}{z}^2}) }{\sqrt{1 - |h_z t z|^2}}
 \s .
\end{align}
It is not difficult to verify that setting $t_1 \propto \ip{h}{b}$
and $t_2 \propto \scrE_p(h)$ in the above display leads to 
the identity
$\scrE_p(h_z t_*) = 0$ with  $t_* = 
\frac{(\ip{h}{b},\scrE_p(h))}{\sqrt{\ip{h}{b}^2
+ \scrE_p^2(h)}}$. Finally, the parameter
\begin{align} \label{tsK}
\s[32]  t_* K =  \frac{t_* \ip{h}{b}}{|\ip{h}{b}|} 
= \frac{(\ip{h}{b}^2, \s \ip{h}{b} \scrE_p(h))}
{\sqrt{\ip{h}{b}^4 + \ip{h}{b}^2 \scrE_p^2(h)}}
\s[32] \Big( K = \frac{\ip{h}{b}}{|\ip{h}{b}|}  \Big)
\end{align}
also has the property that $\scrE_p(h_z t_* K) = 0$
but $t_* K$ admits asymptotic estimators via 
the replacement $\req{gpsq}$ of $\req{gps}$.
This modifies step \ref{gps3} of the GPS program to
find an asymptotic root of $\scrE_p(\s \cdot \s)$ 
without any sign conventions on $\ip{b}{z},\ip{h}{z}$.
While these changes are somewhat trivial for the case 
$q=1$, our
understanding of them is informed by the case $q > 1$
and initiated by the impossibility result in Theorem \ref{thm:angles}.

\subsection{Sketch of the derivation of $\scrH_\sharp$}
\label{sec:sketch}
We begin by defining a matrix $\scrH_z$ 
composed of $(q + 1)$ orthonormal columns,
derived from $\scrH$ in $\req{Jc}$ and
$\zpH$ in $\req{zpH}$, i.e.,  
\begin{align} \label{Hz} 
\scrH_z 
= \big( \scrH \s[8] \zpH \big) 
= \Big( \scrH \s[16] \frac{z-z_\scrH}{|z-z_\scrH|} \Big)  \s ,
\end{align}
so that $\col{\scrH_z}$ expands $\col{\scrH}$ by the vector $z  =
\frac{\zv}{|\zv|}$. We introduce a parametrized
estimator $\scrH_z T$ for a full rank matrix $T$, derive
a root $T_*$ of the map $T \mapsto \scrE_p(\scrH_z T)$,
and construct  an asymptotic estimator of $T_*$
by applying Theorem \ref{thm:pca}\ref{pca:HBBH}--\ref{pca:HBBz}.

We consider the following family of estimators
with $\req{hzt}$ as a special case.
\begin{align} \label{Tfamily}
 \big\{ \s \scrH_z T \cst T \in \bbR^{(q + 1) \times q}
\text{ with } T^\top T \in \bbR^{q \times q} \text{ invertible}
\big\} \s .
\end{align}

Any $\scrH_z T$ in this family is a $p \times q$ matrix of full
rank. We have $T = (t_1 \s[4] t_2)^\top$
for $q =1$, but the constraint on
$ T^\top T = t_1^2 + t_2^2$ imposed by $\req{hzt}$ is relaxed 
in $\req{Tfamily}$. 

Substituting $\scrH_z T$ into the optimization
bias function in $\req{optbias}$, we obtain 
\begin{align} \label{EHT}
\scrE_p(\scrH_z T) = \frac{\scrB^\top z 
- \scrB^\top \scrH_z (TT^\dagger) \scrH_z^\top z}
{1 - |z_{\scrH_z T}|^2} 
\end{align}
where we have used that $\scrH_z^\top \scrH_z = \Id$.
The expression $\req{EHT}$
is obscure, but we note that $T^\top (TT^\dagger) = T^\top$ 
which suggests a simplification post $T^\top_* = 
\scrB^\top \scrH_z$, i.e., 
\begin{align} \label{slick} 
\scrE_p(\scrH_z T_*) 
&= \frac{\scrB^\top z - T_*^\top \scrH_z^\top z}
{1 - |z_{\scrH_z T_*}|^2} 
= \frac{\scrB^\top z - \scrB^\top \scrH_z \scrH_z^\top z}
{1 - |z_{\scrH_z \scrH_z^\top \scrB}|^2}  =  0 \s ,
\end{align}
provided $T^\top_* T_*$
is invertible and $|z_{\scrH_z T_*}| < 1$ (see
Appendix \ref{app:sharp}).
For last equality we use that
$\scrH_z \scrH_z^\top z
 = z_{\scrH_z} = z$ 
(i.e., the projection of $z$ onto $\col{\scrH_z}$ is $z$
itself). We remark that the matrix formalism of $\req{slick}$
has advantages even over the special case 
in $\req{Ehzt}$. Figure \ref{fig:geom} illustrates
geometry of the transformation $T \mapsto \scrH_z T$ at $T = T_*$.

\vspace{-0.04in}
\begin{figure}[htp!]
\centering
\includegraphics[width=5in]{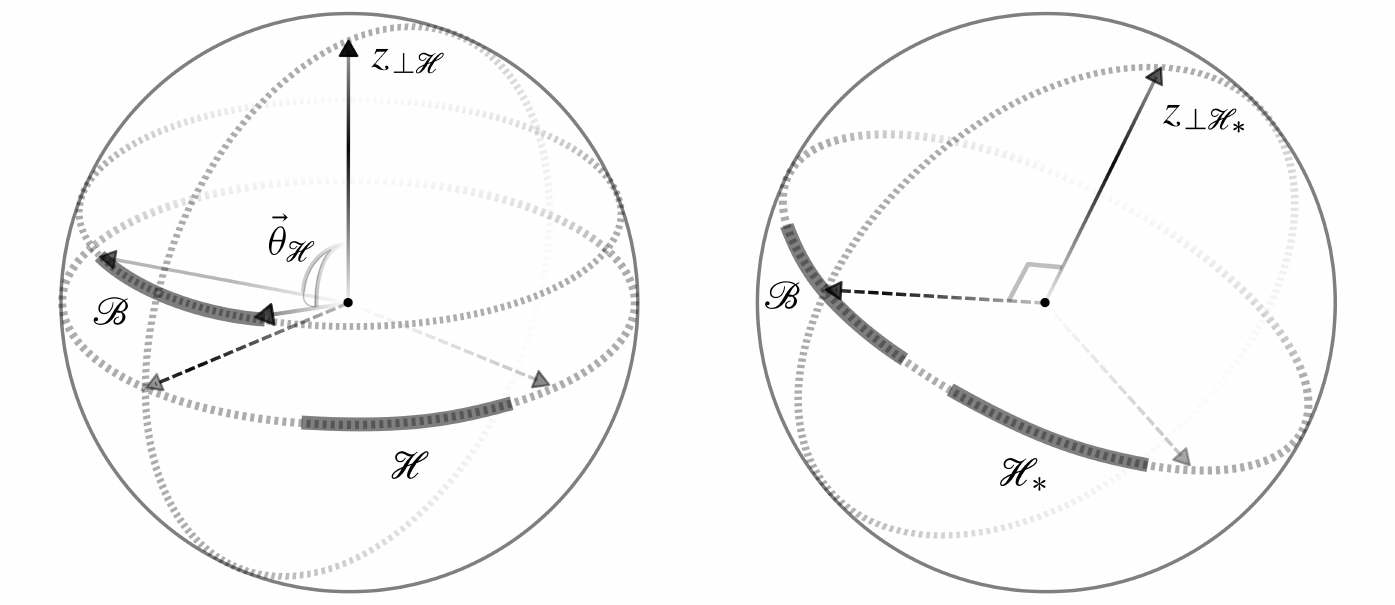}
\vspace{-0.08in}
\caption{\textup{Left panel:} Illustration of $\col{\scrH}$ 
relative to $\zpH$ in $\req{zpH}$. The angles $\vec{\theta}_\scrH$ 
between $\zpH$ and $\col{\scrB}$ are the arc cosines of 
the entries of $\scrE_p(\scrH)
= \scrB^\top \zpH$.
\textup{Right panel:} The optimal
$T_* = \scrH_z^\top \scrB$ leads to a basis 
$\scrH_*$ that spans $\col{\scrH_z T_*}$ and leads to 
$\scrB^\top z_{\perp \scrH_*} = \scrE_p(\scrH_*) = 0$,
setting each entry of $\vec{\theta}_{\scrH_*}$
to $90^\circ$.}
\label{fig:geom}
\end{figure}
\vspace{-0.04in}

The slick calculation above does not constitute our original
derivation which is heavy-handed and superfluous. The advantage
of $\req{slick}$ lies in its brevity and its quick bridge to the
GPS program.  Yet, $\req{slick}$ is not sufficient in view of
Theorem \ref{thm:angles}, i.e.,  the optimal point $T_* =
\scrH_z^\top \scrB$ is not observed nor can it be estimated from
the observed data.  To this end, we seek an invertible matrix
$K$ for which (similarly to $\req{tsK}$), 
\[   T_* K =  \scrH_z^\top \scrB K \]
may be estimated solely from  $\Y$
and use Lemma \ref{lem:K} to conclude
that $\scrE_p(\scrH_z T_* K) = 0$ provided $K$ is invertible.
The choice $K$ is motivated by the fact that as $p \upto \infty$,
\begin{align} \label{HzBBH} 
 \scrH_z^\top \scrB \scrB^\top \scrH 
  = \left( \begin{array}{c}
 \scrH^\top \scrB\scrB^\top \scrH \\ 
  \scrE^\top_p(\scrH) \scrB^\top \scrH
  \end{array} \right)  
  \to \limp \left( \begin{array}{c} \snr \\ 
  \hz^\top
  \end{array} \right)   \snr 
\end{align}
where we used $\scrB \zpH = \scrE_p(\scrH)$ in $\req{optbias_gps}$
and applied Theorem \ref{thm:pca} to obtain the 
stated almost sure convergence (see Appendix \ref{app:sharp}
for details). The variables in the limit 
are computable from the data $\Y$ and it is again
notable 
that while we are unable to estimate $\scrE_p(\scrH)$,
the quantity $\scrE^\top_p(\scrH) (\scrB^\top \scrH)$
admits an estimator as did $|\scrE_p(\scrH)|$.

Our estimator $\req{Hsharp}$ is now easily seen to take the 
following form.
\begin{align} \label{HsharpMat}
\s[32] \scrH_\sharp = \scrH_z T_\sharp  \s ,
\s[32] T_\sharp = 
\Big( \begin{array}{c}
  \snr \\ \hz^\top  \end{array} \Big) \scrM \Phi^{-1} 
\end{align}
This suggests taking $K = \scrB^\top \scrH \snr^{-1}
\Rv \Phi^{-1}$ because $\req{HzBBH}$ now implies that
\begin{align} \label{Tprime}  
\s[32] \limp | T_* K
-  T_\sharp | = 0  \s .
\end{align}

Appendix \ref{app:sharp} proves the $K$ is eventually invertible
and applies $\req{Tprime}$ to deduce that $\scrE_p(\scrH_z T_*
K) = 0$ implies the desired conclusion $\limp \scrE_p(\scrH_z
T_\sharp) =0$.

\section{A Numerical Example}
\label{sec:numerics}

We illustrate our results on a numerical example to
provide a verification of Theorems \ref{thm:pca},
\ref{thm:pcabias} and \ref{thm:Hsharp}. We also study the
convergence rates of various estimators, which are not supplied
by our theory. Consider i.i.d. observations of $y \in
\bbR^{\pu}$ where,
\begin{align} \label{model}
 y = \alpha +  B x + \epsilon  \s ,
\end{align} 
with $\alpha \in \bbR^{\pu}$ and a $\pu \times q$ matrix $B$, that
are realized over uncorrelated $x \in \bbR^q$ and $\epsilon \in
\bbR^{\pu}$ with $\var(x) = \Id_q$ and $\var(\ep) = \Sv$. Then,
$\var(y) = \bSig = BB^\top + \Sv$.  Fixing $n = 120$, $q = 7$, 
$\pu = 128,000$, we simulate  $\pu \times n$ data matrices
$\Y$ with observations of $y$ as its columns.  The parameters
$\alpha, B, \Sv$ 
are calibrated in Section~\ref{sec:model} 
with the minimum variance portfolio problem 
described in Section
\ref{sec:risk} in mind.

We simulate $10^5$ data matrices $\Y_{\pu \times n}$, selecting
subsets $\Y_{p\times n}$ of size $p \times n$ by taking $p =
500, 2000, 8000, 32000, 128000$ to study the asymptotics of
three estimators.  All  three are based on a centered
sample covariance $\Sam = \Y\Jc \Y^\top /n$ (see $\req{Jc}$),
the spectrum of which equals that of $\Jc \Y^\top \Y \Jc / n$
and is computed from this $n \times n$ matrix.  This results in
a $7 \times 7$ diagonal matrix $\qeig^2_p$ with the $7$
largest eigenvalues of $\Sam$, as well as the $\snr^2$ in
$\req{snr-bulk}$ and $\kappa^2_p$ in $\req{bulk2}$. Our three
estimators have the form,
\begin{align} \label{nSig}
\s[32] \hSig_\natural = \scrH_\natural (\qeig_p \snr  )^2 
\scrH^\top_\natural + \hat{\sv}^2 \Id \s ,
\s[16]
(\scrH_\natural \in \{ \scrH, \scrH_\flat, \scrH_\sharp\}),
\end{align}
where $\hat{\sv}^2 = n \kappa_p^2 /\s[-1] p$
and $\scrH_\natural$ 
is one of three $p \times 7$ matrices of orthonormal columns.
\begin{enumerate}[itemsep=0.025in]
\item[($\scrH$)] The sample eigenvectors $\scrH 
=\nu_{p \times 7}(\Sam)$ are computed per $\req{efficient}$
using the matrix  $\Jc \Y^\top \Y \Jc / n$. These vectors
correspond to a {\pca} covariance model in $\req{nSig}$.
\item[($\scrH_\flat$)] The matrix $\scrH_\flat$ will use
the {\gps} estimator of Section \ref{sec:gps} to issue a 
correction to only the first column of $\scrH$. In particular,
we let $\scrK$ equal $\scrH$ in columns $2$--$7$ and replace
its first columns by $h_z t / |h_z t|$ with $h_z t$ in $\req{hzt}$
and $(t_1,t_2)$ given by $\req{tsK}$. Finally, we set
$\scrH_\flat = \nu_{p \times 7}(\scrK \qeig_p \snr)$
computed analogously to $\req{efficient}$
for efficiency.
\item[($\scrH_\sharp$)] The corrected sample
eigenvectors $\scrH_\sharp$ are computed using
$\req{zpH}$--$\req{Hsharp}$.
\end{enumerate}

To assess the performance of the three covariance estimators 
in $\req{nSig}$ we test them on several metrics. With respect
to the minimum variance portfolio application, for  
$\bSig = \var(y)$ and $y \in \bbR^p$ in $\req{model}$,
the returns to $p$ financial assets, we compute
\begin{align}\label{sigmin} 
\s[32] \sigma^2_{\min} = \min_{\ip{1_p}{w} = 1}
\ip{w}{\bSig w}  \s[32] (w \in \bbR^p) \s ,
\end{align}
the true minimum variance. We compare the volatility
$\sigma_{\min}$ to the realized volatility,
$\tru_p(\scrH_\natural) = \sqrt{\ip{\hw_\natural}{\bSig \hw_\natural}}$ 
(see $\req{tru}$) of $\hw_\natural \in \bbR^p$ that minimizes
$\req{minvar}$ with $\hSig = \hSig_\natural$.

We also study the length $|\scrE_p(\scrH_\natural)|$ of
the quadratic optimization bias, 
the true and realized quadratic optima
(taking $\cc=\cl=1$ in $\req{conv}$ and $\req{Qhx}$)
of Section \ref{sec:motive},
\begin{align} \label{numobj}
 \max_{x \in \bbR^p} Q(x) = 1 + \frac{\nz^2_p }{2} \s ,
\s[32]  Q(\hx) = 1 + \frac{\hn^2_p }{2} \mf_p \s ,
\end{align}
and their discrepancy $\mf_p = 2 - \hn_p^2 \tru^2_p$.
The  $\nz^2_p = \ip{1_p}{\bSig^{-1} 1_p}$
and $\hn^2_p = \ip{1_p}{\hSig^{-1}_\natural 1_p}$ (as well as
$\hx_\natural = \hSig_\natural^{-1} 1_p$, and minimizers of
$\req{minvar}$, $\req{sigmin}$)
are efficiently computed via the Woodbury identity
to obtain the inverses of the 
covariance matrices $\bSig$ and $\hSig_\natural$.

\begin{figure}[htp!] 
\centering
\includegraphics[width=5in]{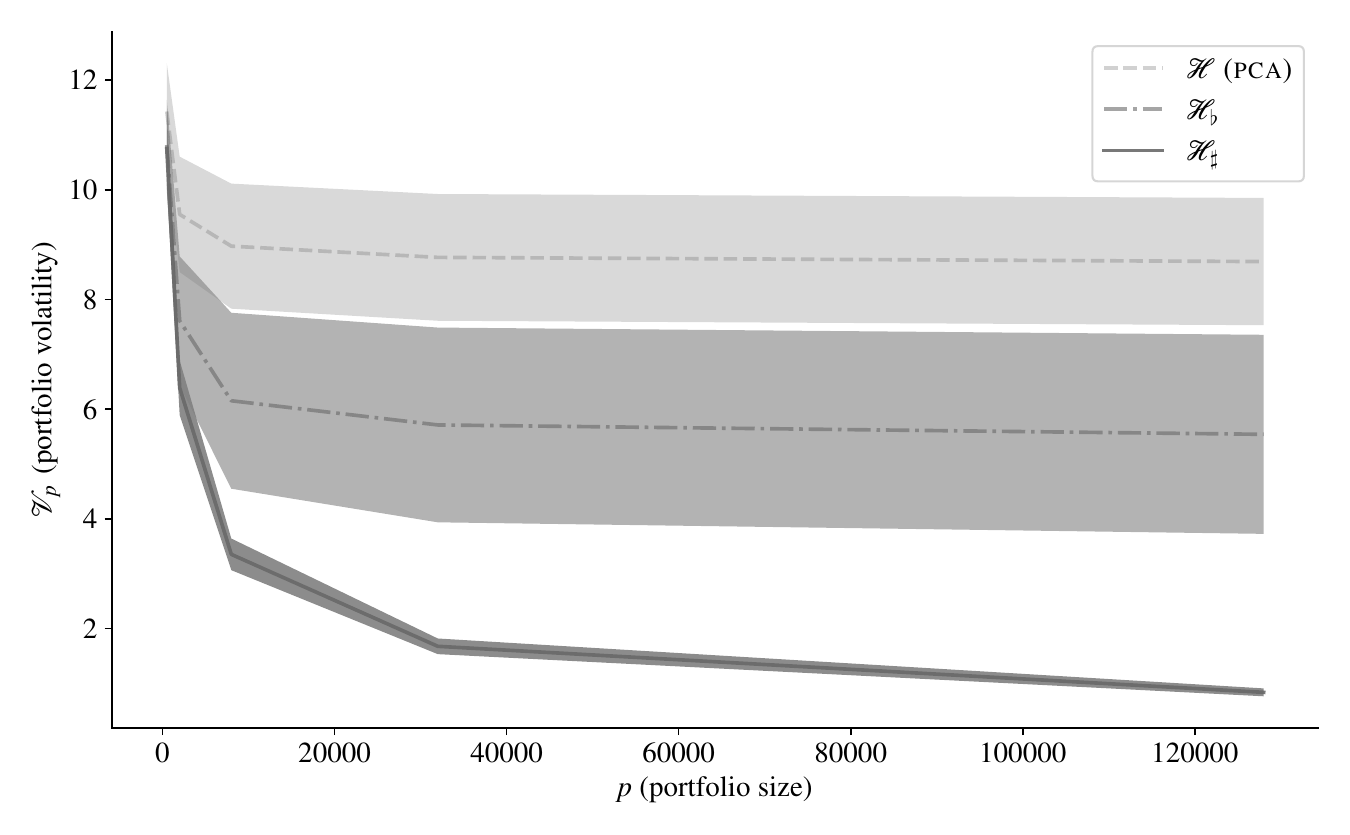}
\caption{Realized minimum variance portfolio volatility
versus portfolio size
$p$ with lines, the sample means of Table \ref{tab:vol},
and two standard deviation confidence intervals.}
\label{fig:vol}
\end{figure}

\begin{figure}[htp!] 
\centering
\includegraphics[width=5in]{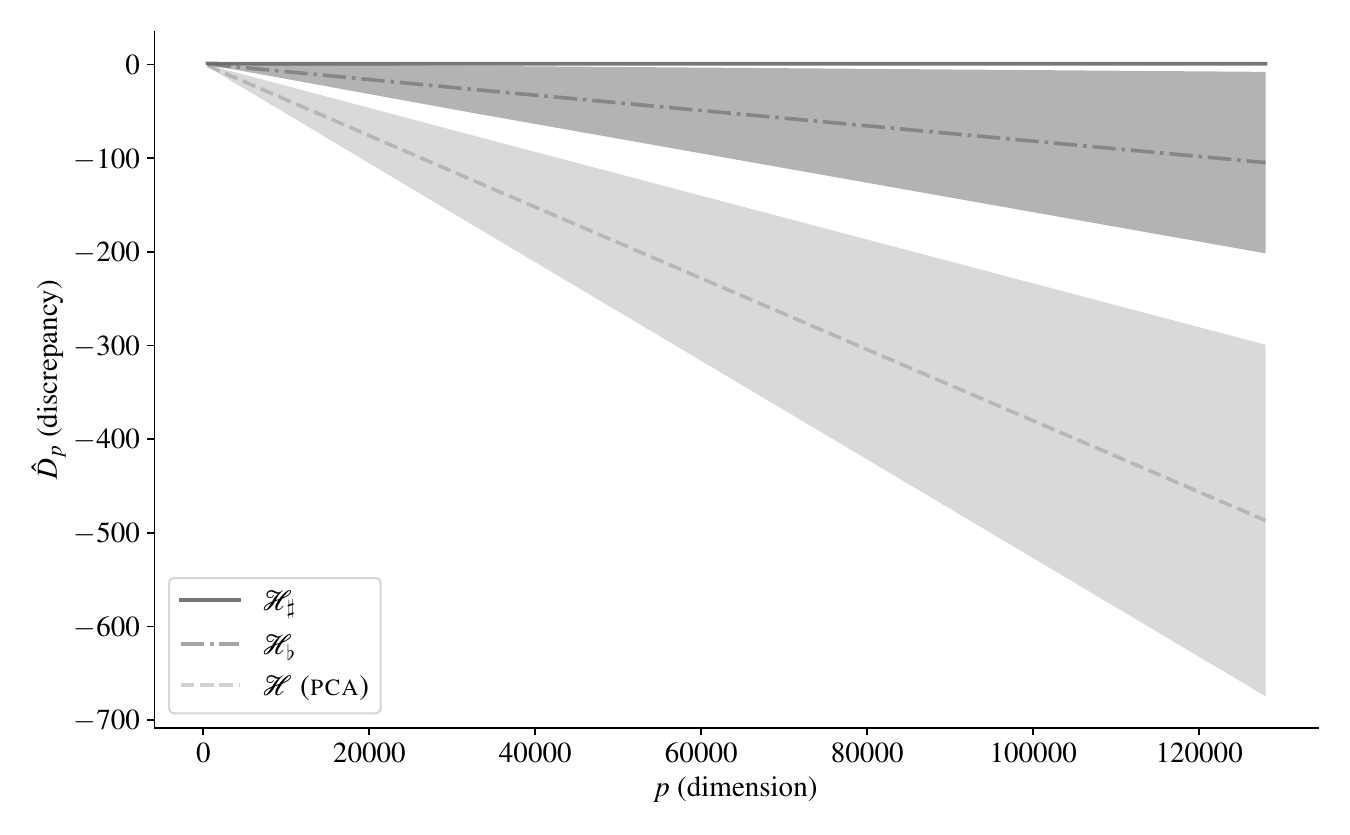}
\caption{Discrepancy $\mf_p$ versus dimension $p$ with 
lines, the sample means found in Table \ref{tab:discrep} and
Table \ref{tab:gps}, and two standard deviation confidence
intervals, (the confidence intervals for the estimator
$\scrH_\sharp$ are too narrow to be clearly visible).}
\label{fig:discrep}
\end{figure}

\begin{table} 
\begin{center} 
\begin{tabular}{rcccc} 
\toprule 
$p$ & $\sigma_{\min}$ & $\Exp \s \tru_p(\scrH)$ & $\Exp \s \tru_p(\scrH_{\flat})$ & $\Exp \s \tru_p(\scrH_{\sharp})$ \\ 
\midrule 
$500$ & $7.69$ & $11.43$ & $10.82$ & $10.75$ \\ 
$2000$ & $4.09$ & $9.55$ & $7.61$ & $6.39$ \\ 
$8000$ & $2.08$ & $8.97$ & $6.15$ & $3.35$ \\ 
$32000$ & $1.04$ & $8.77$ & $5.71$ & $1.68$ \\ 
$128000$ & $0.52$ & $8.69$ & $5.54$ & $0.84$ \\ 
\bottomrule 
\end{tabular} 
\end{center} 
\caption{Realized minimum variance portfolio volatilities (square root of $\tru_p^2$) computed using the three covariance estimators ${\hSig}$ ({\pca}), ${\hSig}_{\flat}$ and ${\hSig}_{\sharp}$. Sample mean estimates for expectation $\Exp \s X$ of column variable $X$ based on $10^5$ simulations.} 
\label{tab:vol} 
\end{table}

\begin{table} 
\begin{center} 
\begin{tabular}{rcrrcc} 
\toprule 
$p$ & $\max_x \s[-1] Q(x)$ & $\Exp \s Q(\hat{x})$ & $\Exp \s \mf_p(\scrH)$ & $\Exp \s Q(\hat{x}_{\sharp})$ & $\Exp \s \mf_p(\scrH_{\sharp})$ \\ 
\midrule 
$500$ & $1.01$ & $0.99$ & $-1.16$ & $1.0$ & $1.22$ \\ 
$2000$ & $1.03$ & $0.64$ & $-7.11$ & $1.01$ & $0.93$ \\ 
$8000$ & $1.12$ & $-4.98$ & $-30.04$ & $1.04$ & $0.85$ \\ 
$32000$ & $1.47$ & $-97.01$ & $-121.81$ & $1.18$ & $0.86$ \\ 
$128000$ & $2.88$ & $-1572.9$ & $-486.92$ & $1.7$ & $0.87$ \\ 
\bottomrule 
\end{tabular} 
\end{center} 
\caption{Realized maximum $Q(\hx)$ and discrepancy $\mf_p(\scrH)$ of {\pca} for growing $p$ are compared with $Q(\hx_{\sharp})$ and  $\mf_p(\scrH_{\sharp})$, computed using the covariance $\hSig_{\sharp}$. The true maximum in column $2$ applies the true covariance matrix $\bSig$. Sample mean estimates for expectation $\Exp \s X$ of column variable $X$ based on $10^5$ simulations.} 
\label{tab:discrep} 
\end{table}

\begin{table} 
\begin{center} 
\begin{tabular}{rcccrr} 
\toprule 
$p$ & $\Exp \s |\hz|$ & $\Exp \s |\scrE_p(\scrH)|$ & $\Exp \s
|\scrE_p(\scrH_{\sharp})|$ & $p \s \Exp \s |\scrE_p(\scrH)|^2$ &
$p \s \Exp \s |\scrE_p(\scrH_{\sharp})|^2$ \\ 
\midrule 
$500$ & $0.237$ & $0.298$ & $0.19$& $44.5$ & $18.05$ \\ 
$2000$ & $0.219$ & $0.26$ & $0.132$& $135.6$ & $34.88$ \\ 
$8000$ & $0.222$ & $0.237$ & $0.051$& $447.8$ & $21.06$ \\ 
$32000$ & $0.225$ & $0.228$ & $0.018$& $1663.9$ & $10.03$ \\ 
$128000$ & $0.227$ & $0.228$ & $0.008$& $6629.1$ & $7.68$ \\ 
\bottomrule 
\end{tabular} 
\end{center} 
\caption{Quadratic optimization bias length $|\scrE_p(\scrH)|$ for {\pca}, its estimator $|\hz|$ of Theorem \ref{thm:pcabias} and $|\scrE_p(\scrH_{\sharp})|$ are shown for growing $p$. Scaled variables $p \s |\scrE_p(\scrH)|^2$ and $p \s |\scrE_p(\scrH_{\sharp})|^2$ are provided to illustrate the convergence rates. Sample mean estimates for expectation $\Exp \s X$ of column variable $X$ based on $10^5$ simulations.} 
\label{tab:optbias} 
\end{table}

\begin{table} 
\begin{center} 
\begin{tabular}{rcrrrr} 
\toprule 
$p$ & $\max_x \s[-1] Q(x)$ & $\Exp \s Q(\hat{x}_{\flat})$ & $\Exp \s \mf_p(\scrH_{\flat})$ & $\Exp \s |\scrE_p(\scrH_{\flat})|$ & $p \s \Exp \s |\scrE_p(\scrH_{\flat})|^2$ \\ 
\midrule 
$500$ & $1.01$ & $1.0$ & $0.52$  & $0.278$ & $38.7$ \\ 
$2000$ & $1.03$ & $0.97$ & $-0.97$  & $0.236$ & $111.8$ \\ 
$8000$ & $1.12$ & $0.37$ & $-5.92$  & $0.211$ & $354.5$ \\ 
$32000$ & $1.47$ & $-10.27$ & $-26.19$  & $0.201$ & $1292.0$ \\ 
$128000$ & $2.88$ & $-179.42$ & $-104.76$  & $0.2$ & $5138.0$ \\ 
\bottomrule 
\end{tabular} 
\end{center} 
\caption{Discrepancy and optimization bias metrics reported in Tables \ref{tab:optbias} \& \ref{tab:discrep} recomputed with the vectors $\scrH_{\flat}$ and the corresponding covariance estimator. Sample mean estimates for expectation $\Exp \s X$ of column variable $X$ based on $10^5$ simulations.} 
\label{tab:gps} 
\end{table}

\begin{table} 
\begin{center} 
\begin{tabular}{rcccc} 
\toprule 
$p$ & $\Exp\s |\scrH^\top_{\sharp}\scrB\scrB^\top\scrH_{\sharp} - \Phi^2|$ & $\Exp \s |\Phi^2|$ & $\Exp \s |\scrH^\top \scrB\scrB^\top\scrH - \snr^2|$ & $\Exp \s |\snr^2|$ \\ 
\midrule 
$500$ & $0.502$ & $0.9805$ & $0.5023$ & $0.962$ \\ 
$2000$ & $0.2938$ & $0.9807$ & $0.2952$ & $0.9661$ \\ 
$8000$ & $0.0839$ & $0.9812$ & $0.0875$ & $0.9677$ \\ 
$32000$ & $0.0251$ & $0.9812$ & $0.0262$ & $0.9679$ \\ 
$128000$ & $0.0096$ & $0.9811$ & $0.0099$ & $0.9678$ \\ 
\bottomrule 
\end{tabular} 
\end{center} 
\caption{The inner products of the columns of $\scrH$ ({\pca}) and $\scrH_{\sharp}$ after projection onto $\col{\scrB}$, and  their estimators $\snr^2$ and $\Phi^2$. The norms $|\snr^2|$ and $|\Phi^2|$ estimate the largest, squared projected lengths of the columns of $\scrH$ and $\scrH_{\sharp}$ respectively. Sample mean estimates for expectation $\Exp \s X$ of column variable $X$ based on $10^5$ simulations.} 
\label{tab:HB} 
\end{table}

\subsection{Discussion of the results} \label{sec:results} Table
\ref{tab:vol} and Figure \ref{fig:vol} summarize the simulations
for our minimum variance portfolio application. Volatilities are
quoted in percent annualized units (see Section \ref{sec:model})
and only portfolio sizes $p =500, 2000, 8000$ should be
considered as practically relevant. Three sets of  portfolio
weights ($\hw, \hw_{\flat}, \hw_\sharp \in \bbR^p$) constructed
with the covariance models in $\req{nSig}$ are tested.  As
predicted by $\req{tru_asymp}$, the realized portfolio
volatility $\tru_p(\scrH)$ for the {\pca}-model weights $\hw$ in
the third column of Table \ref{tab:vol} remains bounded away
from zero (on average).  The same holds for the partially
corrected estimator $\scrH_{\flat}$, which substantially
decreases the volatility of the {\pca} weights for larger
portfolios. This estimator was also tested for $q = 4$ in
\ci{goldberg2020}, but for a model in which $(\scrB^\top z)_j
\to 0$ as $p \upto \infty$ for $j = 2,3,4$. In this special
case, the estimators $\scrH_{\flat}$ and $\scrH_\sharp$ coincide
asymptotically.  In our more realistic model of Section
\ref{sec:model}, all sample eigenvectors require correction as
evident by comparing $\tru_p(\scrH_\flat)$ and
$\tru_p(\scrH_\sharp)$ in Table \ref{tab:vol}. The latter
portfolio volatility decays at the rate of roughly $1/\s[-2]
\sqrt{p}$. The true volatility $\sigma_{\min}$ (second column of
Table \ref{tab:vol}) also decays at this rate.
Figure~\ref{fig:vol} depicts the much larger deviations about
the average that the estimators $\scrH$ and $\scrH_\flat$
produces on the portfolio volatility metric relative to
$\scrH_\sharp$. Surprisingly, $\scrH_{\flat}$ produced the
largest such deviations.

Table \ref{tab:discrep} and Figure \ref{fig:discrep} compare the
{\pca}-model $\hSig$ to our optimization-bias free estimator
$\hSig_\sharp$ on the quadratic function objectives in
$\req{numobj}$. As predicted in Section~\ref{sec:motive} and
Theorem \ref{thm:discrepancy} in particular, the true objective
value (the second column of Table~\ref{tab:discrep}) increases
in $p$ while the realized objective $Q(\hx)$ decreases rapidly.
The (expected) discrepancy $\mf_p(\scrH)$ of the {\pca}-model is
shown to diverge to negative infinity linearly with the
dimension as predicted by Theorem \ref{thm:discrepancy} (i.e,
largest $q$ eigenvalues of the covariance model of Section
\ref{sec:model} diverge in $p$). The last two columns of Table
\ref{tab:discrep} confirm the realized maximum and discrepancy
produced by the corrected eigenvectors $\scrH_\sharp$ behave in
a more desirable way.  The discrepancy $\mf_p(\scrH_\sharp)$
appears to converge in a neighborhood of the optimal value one,
while the realized maximum $Q(\hx_\sharp)$ has a trend similar
to that of the true maximum.  Figure \ref{fig:discrep} shows the
large uncertainly of the average behaviour summarized in Table
\ref{tab:discrep} that results from using the sample
eigenvectors $\scrH$ or their partially corrected version,
$\scrH_\flat$ (see Table \ref{tab:gps} for the averages).  The
uncertainty produced by the corrected eigenvectors
$\scrH_\sharp$ is negligible by comparison. 

Table \ref{tab:optbias} summarizes our numerical results on the
length of the quadratic optimization bias $|\scrE_p(\s \cdot
\s)|$ for the sample eigenvectors $\scrH$ and the corrected
vectors $\scrH_\sharp$. Table \ref{tab:gps} supplies the same
for the partially corrected eigenvectors $\scrH_\flat$. The
first three columns of Table \ref{tab:optbias} confirm the
findings of Theorem \ref{thm:pcabias}, i.e., the length
optimization bias $|\scrE_p(\scrH)|$ for {\pca} may be
accurately estimated from observable data in higher dimensions.
We find that the expected length $|\scrE_p(\scrH)|$ converges
away from zero, and that $ p \s |\scrE_p(\scrH)|^2$ diverges in
expectation.  This is predicted by Theorem \ref{thm:pcabias}
since $\scrH^\top z \in \bbR^7$ does not vanish as $p$ grows.
Table \ref{tab:gps} presents similar findings for $\scrH_\flat$,
which we have not analyzed theoretically.  Column $4$ of
Table~\ref{tab:optbias} confirms the predictions of Theorem
\ref{thm:Hsharp}, i.e., the corrected bias length
$|\scrE_p(\scrH_\sharp)|$ vanishes as $p$ grows. Our numerical
finding expand on this by also demonstrating that $p \s
|\scrE_p(\scrH_\sharp)|^2$ appears to be bounded (in
expectation). This suggest a convergence rate of $O(1 / \s[-2]
\sqrt{p})$ for the corrected bias $|\scrE_p(\scrH_\sharp)|$.
The latter is consistent with the asymptotics of
$\mf_p(\scrH_\sharp)$ in Table~\ref{tab:discrep} which Theorem
\ref{thm:discrepancy} forecasts to behave as $O(1) ( 1 - p \s
|\scrE_p(\scrH_\sharp)|^2)$.

Table \ref{tab:HB} provides support for 
Theorem \ref{thm:pca}\ref{pca:HBBH} and Theorem \ref{thm:Hsharp} 
which concerns the projection of the estimated eigenvectors
onto the population subspace $\col{\scrB}$. The
convergence verified in columns two and four show that
the vectors in $\scrH$ and $\scrH_\sharp$ remain
orthogonal after projection onto $\col{\scrB}$ because
$\Phi^2$ and $\snr^2$ are diagonal matrices. The largest
elements of these matrices (presented as averages
in columns three and five) estimate the largest length
squared of the columns of $\scrH$ and $\scrH_\sharp$ 
in $\col{\scrB}$
respectively. This confirms $\scrH_\sharp$ has a larger
such projection than does $\scrH$.

\begin{figure}[htp!] 
\centering
\includegraphics[width=2.75in]{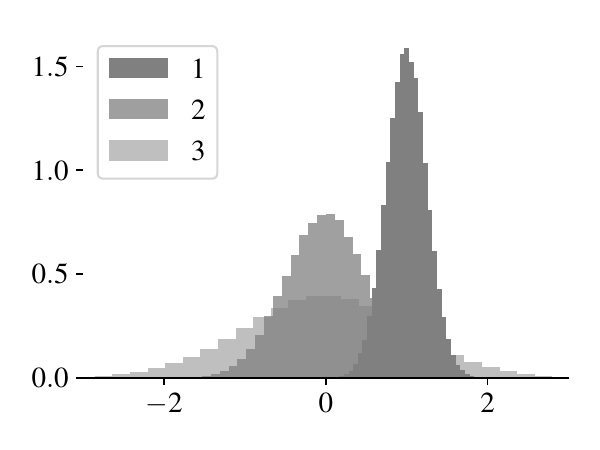}
\includegraphics[width=2.75in]{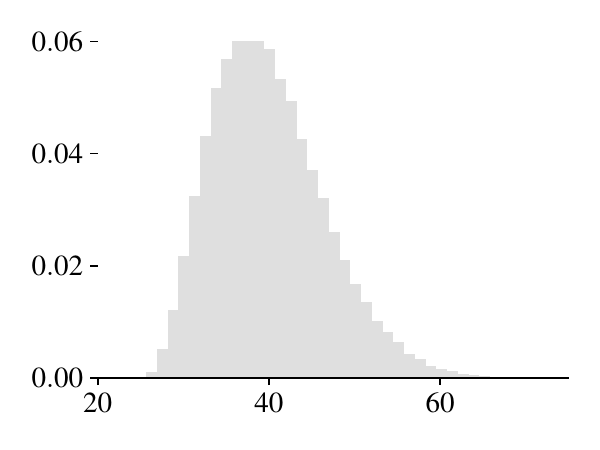}
\caption{\textup{Left panel:} Histograms of 
exposures to the market and style risk factors (i.e., 
the first three columns of the risk factor exposures matrix
$\Xi$). \textup{Right panel:} Histogram of 
asset specific volatilities (i.e., square roots of
the diagonal entries of $\Sv$).}
\label{fig:hist}
\end{figure}

\begin{figure}[htp!]
\centering
\includegraphics[width=2.75in]{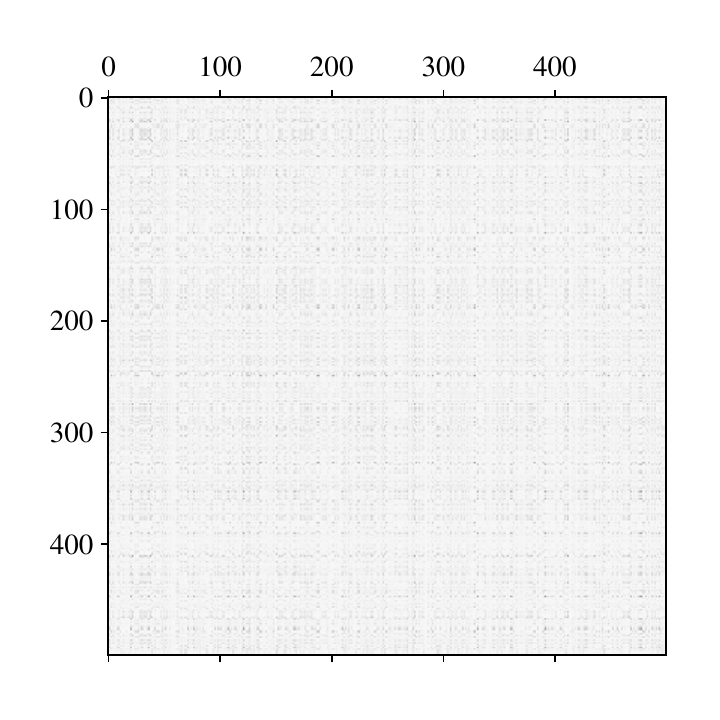}
\includegraphics[width=2.75in]{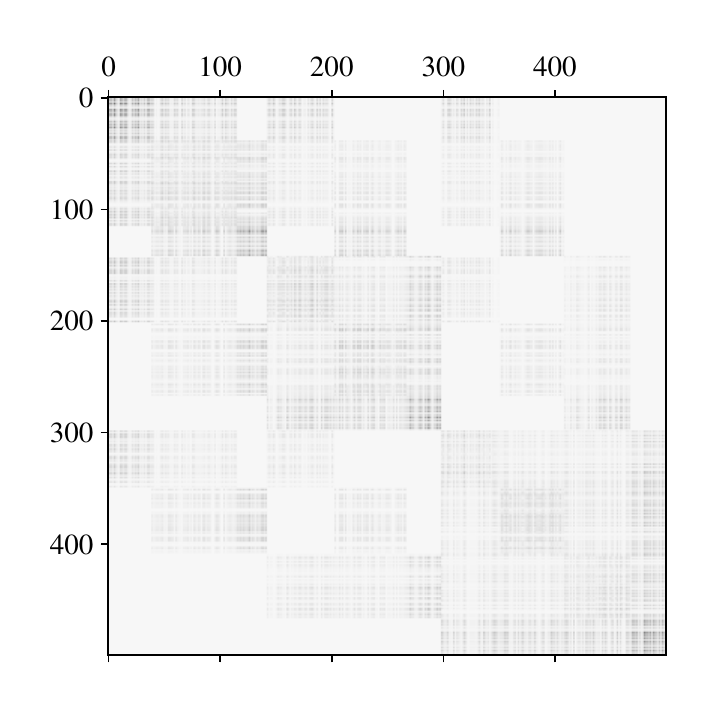}
\caption{Visualization of the asset industry memberships matrix
$\calM = \sum_{j=4}^7 \xi_j \xi^\top_j$ for $\xi_j$, the
$j$th column of the risk factor exposure matrix $\Xi$ (for a
random sample of $500$ assets). The left panel shows $\calM$ and
the right panel shows its Cuthill–McKee ordering. The
block structure corresponds to industry groupings.}
\label{fig:industry}
\end{figure}

\subsection{Population covariance model}
\label{sec:model}
Our covariance matrix calibration loosely follows the specification of 
the Barra US equity risk model (see \ci{menchero2011}
and \ci{blin2022}). To this end, we introduce a (random) vector
of factor returns $f \in \bbR^7$ and a $\pu \times 7$ exposure
matrix $\Xi$  which satisfy,
\begin{align}  \label{varf}
 \var(f) = \left( \begin{array}{ccccccc} 
$250$ & $0$ & $0$ & $55$ & $44$ & $68$ & $-22$ \\ 
$0$ & $64$ & $0$ & $0$ & $0$ & $0$ & $0$ \\ 
$0$ & $0$ & $16$ & $0$ & $0$ & $0$ & $0$ \\ 
$55$ & $0$ & $0$ & $481$ & $192$ & $-108$ & $0$ \\ 
$44$ & $0$ & $0$ & $192$ & $260$ & $-8$ & $22$ \\ 
$68$ & $0$ & $0$ & $-108$ & $-8$ & $160$ & $-44$ \\ 
$-22$ & $0$ & $0$ & $0$ & $22$ & $-44$ & $121$ \\ 
\end{array} \right)  \s , \s[16] Bx = \Xi f \s ,
 \end{align}
with $B x$ in $\req{model}$ such that $f = A x$ and $\var(f) =
AA^\top$.  The factor returns  $f$ are Gaussian with mean-zero
and covariance in $\req{varf}$. The unit are chosen so that
the factor volatilities (square roots of the diagonal of 
$\var(f)$) are in units of annualized percent. The columns of $\Xi$ are exposures
to ($q = 7$) fundamental risk factors (market risk, two style
risk factors and fours industry risk factors), and are
generated as follows.
\begin{enumerate}[itemsep=0.04in, leftmargin=0.32in]
\item[--] The entries of the first column (exposures to market risk) of 
$\Xi$  are drawn as i.i.d. normal with mean $1.0$ and standard 
deviation $0.25$.
The second and third columns of $\Xi$
(style risk factors) have i.i.d. entries  that are normal 
with mean zero and standard deviations $0.5$ and $1.0$
respectively for those columns.
\item[--] The last four columns of $\Xi$ are initialized to
be zero and for each row $i$, independently of all other rows, we select two industries
$I_1$ and $I_2$ from $\{1, 2, 3, 4\}$ uniformly at
random and without replacement. Then, drawing $U_1$ and $U_2$
that are independent and uniform in $(0,1)$, we set 
$\Xi_{iI_1} = U_1$ and 
$\Xi_{iI_2} = \Xi_{i I_1} + U_2$.
\end{enumerate}
The left panel of Figure \ref{fig:hist} contains histograms of
the first three columns of $\Xi$.  This calibration of market
and style risk factors is similar to that in \ci{goldberg2020},
who do not consider industry risk, and compare the estimators
$\scrH$ and $\scrH_\flat$ in simulation.  The entries of the
last four columns, which correspond to industry risk factors,
have the following interpretation. Each asset chooses two
industries for membership with an exposure of $0.5$ to each on
average.  When the chosen industries are the same, that exposure
is $1.0$ on average (i.e., $U_1 + U_2$).
Figure~\ref{fig:industry} supplies a visual illustration of the
structure of these industry memberships.  The industry risk
factors drive the poor performance of the estimator
$\scrH_\flat$ in our simulations due to the nonzero projection
that the corresponding four columns of $\Xi$ have in $\col{z}$.
The latter translates to components of the optimization bias
vector $\scrE_p(\scrH_\flat)$ that materially deviate from zero,
and the first that is suboptimally corrected.

The asset specific return $\ep \in \bbR^{\pu}$ in $\req{model}$
are drawn from a mean-zero  Gaussian distribution with a
diagonal covariance matrix $\var(\ep) = \Sv$. We take $\Sv_{ii}
= \sv_i^2$, for asset specific volatilities $\sv_i$, drawn as
independent copies of  $25 + 75 \times Z$ where $Z$ is a
$\text{Beta}(4,16)$ distributed random variable. These are
quoted in annualized percent units, and we refer the reader to
\ci{clarke2011} for typical values that are estimated in
practice.  Lastly, the expected return vector $\alpha \in
\bbR^{\pu}$ in $\req{model}$ is taken as $\alpha = \Xi \sigma_f$
for $\sigma_f = \sqrt{\diag (\var(f))} \approx (15.81,
8, 4, 21.93, 16.12, 12.65, 11)$.


\begin{appendix}
\section{Proofs for Section \ref{sec:qob}}
\label{app:qob}

By direct computation based on the definition in $\req{eH}$ we obtain,
\begin{align} \label{zzH}
\ip{z}{z_H} 
= \ip{H^\top z}{(H^\top H)^{-1} H^\top z}
= \ip{HH^\dagger z}{  H H^\dagger z} = |z_H|^2 
\end{align}
and, recalling that $z = \frac{\zv}{|\zv|}$, the above yields the 
following useful identities.
\begin{align} \label{zmzH}
|z - z_H|^2 = 1 -  |z_H|^2 = \ip{z}{z - z_H} = 1 - |\zv_H|^2/|\zv|^2
\end{align}

The right side is bounded away from zero in $p$ under Assumption
\ref{asm:eH}.  Throughout, we regard $\hat{\sv}$ as a sequence
in $p$ that is bounded in $(0,\infty)$.  We also introduce an
auxiliary sequence $\rv = \rv_p \upto \infty$ to generalize the
rates in Assumptions~\ref{asm:afm} and \ref{asm:eH} so that both
$B^\top B /\rv_p$ and $H^\top H /\rv_p$ converge to invertible
$q \times q$ matrices.

We begin by expanding on some of the calculations
in Section~\ref{sec:motive} and Section~\ref{sec:qob}.
Starting with $\req{hQx}$, the maximizer of $\hat{Q}(\s \cdot
\s)$ is easily calculated as $\hx = \cl \hSig^{-1} \zv$, and
\begin{align*} 
\s[32]  \max_{x \in \bbR^p} \hat{Q}(x) 
&= \cc + \cl^2 \s \ip{\zv}{\hSig^{-1} \zv}
 - \frac{\cl^2}{2} \s \ip{\zv}{\hSig^{-1} \zv} 
\\&= \cc + \frac{\cl^2 \hn^2_p}{2}
\end{align*}
justifying the expression for $\hat{Q}(\hx)$ below
$\req{hQx}$  with $\hn_p^2 = \ip{\zv}{\hSig^{-1} \zv}$
as well as $\req{conv}$.

Define $\hw =  \hSig^{-1} \zv\s / \hn_p^2  = \hx \s / (\cl \hn_p^2)$ 
and set $\tru^2_p = \ip{\hw}{\bSig \hw}$ per $\req{tru}$. Then,
\begin{align*}
{Q}(\hx) 
&= \cc + \cl^2 \s \hn_p^2 -
\frac{1}{2}\ip{\hx}{\bSig \hx}  
\\&= \cc + \frac{\cl^2 \s \hn^2_p}{2} \Big( 2 
 -  \frac{\ip{\hx}{\bSig \hx}}{ \cl^2 \hn^2_p }\Big) \\
&= \cc + \frac{\cl^2\s \hn^2_p}{2} \s \big(
 2 -  \hn^2_p \tru_p^2 \big)
\end{align*}
which is identical to $\req{QhxD}$ with 
$\mf_p =  2 -\hn^2_p \s \tru_p^2$.

Lastly, we recognize $\hw$ as the (unique) solution of $\req{minvar}$. 
The following provides a useful decomposition of 
these solutions. 

\begin{lemma} \label{lem:hw}
Suppose $H = H_{p \times q}$ has $\limp H^\top H /\rv_p$ as an
invertible $q \times q$ matrix.  Then, for vectors $v \in
\bbR^p$ with $\sup_p |v| < \infty$, the minimizer $\hw$ of
$\req{minvar}$ has
\begin{equation} \label{hwu}
    \hw 
= \frac{\zv-\zv_H}{\ip{\zv}{\zv - \zv_H}} 
+ \frac{v}{|\zv| \rv_p} 
= \frac{1}{|\zv|} \Big( \frac{z-z_H}{\s |z - z_H|^2} 
+ \frac{v}{  \rv_p }  \Big)\s .
\end{equation}
\end{lemma}

\begin{proof}
 We begin with an expression of $\hSig^{-1}$ 
via the Woodbury identity. 
\begin{align*}
        \hSig^{-1} = \frac{1}{\hat{\sv}^{2}}\s
   \Big(\Id_p -  \hat{\sv}^{-2} H \big(\Id_q  + \hat{\sv}^{-2}
 H^{\top}H \big)^{-1} H^{\top}  \Big)
\end{align*}
Next, consider the singular value decomposition $H = \scrH
\qeig_p \scrU^\top$ where $\scrH \in \bbR^{p\times q}$
and  $\scrU \in
\mathbb{R}^{q\times q}$ have orthonormal columns and $\qeig_p \in
\mathbb{R}^{q\times q}$ is diagonal.
Then,
\begin{align}
    \hSig^{-1} 
&=\frac{1}{\hat{\sv}^2} \s \big( 
  \Id_p - \scrH \qeig_p \s \big(\hat{\sv}^2 \Id_q +\qeig^2_p\big)^{-1}
  \qeig_p \scrH^{\top} \big)   \notag
\\&=\frac{1}{\hat{\sv}^2} \s \big( 
  \Id_p - \scrH\scrH^{\top} + \hat{\sv}^2\scrH 
  \big(\hat{\sv}^2 \Id_q +\qeig_p^2\big)^{-1}\scrH^{\top} \big)
  \label{eq:portcalc1}
\end{align}
where at the last step we utilized that
$\frac{d^2}{\hat{\sv}^2 + d^2} = 1 - 
\frac{\hat{\sv}^2}{\hat{\sv}^2 + d^2}$ 
and that $\qeig_p$ is diagonal. 

Starting with the expression $\hw = \hSig^{-1} \zv \s /\s \hn_p^2$, 
we define
\begin{align} \label{Cp}
C_p =\hat{\sv}^2 \s  \scrH
( \hat{\sv}^2 \Id + \qeig_p^2)^{-1}\scrH^{\top} z  \s ,
\end{align}
substitute $\req{eq:portcalc1}$ and use that
$\frac{1}{1 + \delta} = 1 - \frac{\delta}{1 + \delta}$
and $\zv_H = HH^\dagger \zv = \scrH \scrH^\top \zv$, to obtain
\begin{align*}
    \hw  
   &= \frac{\zv-\zv_H + |\zv| \s C_p}
   {\ip{\zv}{\zv-\zv_H} + |\zv| \ip{\zv}{C_p} } 
   = \Big( \frac{\zv-\zv_H}{ \ip{\zv}{\zv-\zv_H} } + 
  \frac{|\zv| \s C_p}{\ip{\zv}{\zv-\zv_H} }  \Big) \Big(
   1 - \frac{\delta_p}{1 + \delta_p} \Big) 
\end{align*}
for $\delta_p = 
\frac{|\zv| \ip{\zv}{C_p}}{  \ip{\zv}{\zv-\zv_H}}
= \frac{\ip{z}{C_p}}{|z -z_H|^2} $
per $\req{zmzH}$.
This identifies $v$ in $\req{hwu}$ via the relation,
\begin{align*}  
\frac{v}{|\zv| \rv_p}  &= 
\frac{|\zv| \s C_p}{\ip{\zv}{\zv-\zv_H} }  \Big(
 1 - \frac{\delta_p}{1 + \delta_p}  \Big)
- \frac{\zv-\zv_H}{ \ip{\zv}{\zv-\zv_H} }
\Big(\frac{\delta_p}{1 + \delta_p}  \Big)\\
 &= \frac{|\zv| \s C_p}{\ip{\zv}{\zv-\zv_H} }  \Big(
 \frac{1}{1 + \delta_p}  \Big)
- \frac{\zv-\zv_H}{ \ip{\zv}{\zv-\zv_H} }
\Big(\frac{\delta_p}{1 + \delta_p}  \Big)
\\ &= \frac{|\zv| C_p - (\zv-\zv_H)\s \delta_p}{
(1 + \delta_p)\ip{\zv}{\zv-\zv_H} }  
\\&= \frac{1}{|\zv|} 
\bigg( \frac{C_p  -  (z-z_H)\s \delta_p   }{
(1 + \delta_p) \s |z-z_H|^2) }   \bigg)
\end{align*}

Because $0 \le \delta_p  \le |C_p|\s/ |z - z_H|^2$, to conclude
the proof, it now suffices to show that $|C_p|$ is $O(1/\rv_p)$
so that also $\sup_p |v| < \infty$ as required. We have,
    \begin{align}
       |C_p| &\leq \hat{\sv}^2 
|\scrH (\hat{\sv}^2 \Id +\qeig_p^2)^{-1}\scrH^{\top}|  \notag \\ 
        &= \max_{j \le q} \s 
  \frac{\hat{\sv}^2  }{(\hat{\sv}^2 \Id +\qeig^2_p)_{jj}} \notag\\
   &\leq \max_{j \le q} \s \hat{\sv}^2  / (\qeig_p^2)_{jj} \notag
\\&= \hat{\sv}^2 \s |\scrU \qeig_p^{-2} \scrU^\top | \notag
\\& = (\hat{\sv}^2 /\rv_p) \s |(H^{\top}H)^{-1} \rv_p| \s . 
\label{eq: portcalc2}
\end{align}

Since the spectral norm $| \s \cdot \s |$ and the inverse of a matrix 
over invertible matrices are continuous functions, our 
assumption on $\limp H^\top H/\rv_p$
implies that $\rv_p\s |(H^{\top}H)^{-1} |$ converges to a finite
number. This, together with $\req{eq: portcalc2}$ completes the
proof.  \end{proof}

\begin{proof}[Proof of Theorem \ref{thm:discrepancy}]
Continuing from the expression $\mf_p =  2 -\hn^2_p \s
\tru_p^2$, we first address the asymptotics of $\tru_p^2
= \ip{\hw}{\bSig\hw}$, and  
using $\req{bSig}$ this yields,
\begin{equation}
     \scrV_p^2 
  = |B^{\top} \hw|^2 + \ip{\hw}{\Sv \hw}. 
 \label{V2}
\end{equation}
Applying $\req{hwu}$ of Lemma \ref{lem:hw} and the positive definiteness
of $\Sv$ per Assumption \ref{asm:afm}, 
\begin{align} \label{wGw}
   0 \leq \ip{\hw}{\Sv \hw} 
= \frac{1}{|\zv|^2} \Big(
\frac{\ip{u_H}{\Sv u_H}}{ |z - z_H|^2} 
+ \frac{2\ip{v}{\Sv u_H} }{\rv_p |z - z_H|}
+ \frac{\ip{v}{\Sv v}}{ \rv_p^2} \Big)
\end{align}
where $u_H = \frac{z-z_H}{|z-z_H|}$ and $\sup_p |v| < \infty$.
Turning our attention to the first term in $\req{V2}$ by letting 
$\hw_H = 
\frac{\zv-\zv_H}{\ip{\zv}{\zv-\zv_H}}
=\frac{1}{|\zv|} \frac{z- z_H}{|z - z_H|^2}$,
we again apply Lemma \ref{lem:hw} to deduce that
\begin{align*}
   |B^{\top} \hw|^2
    &=\Big|B^\top \big(\hw_H  + \frac{v}{|\zv| \rv_p}\big)\Big|^2
    = |B^\top \hw_H|^2 + 
2 \frac{\ip{B^\top \hw_H}{B^\top v}}{|\zv| \rv_p}  
+ \bigg( \frac{|B^\top v|}{|\zv| \rv_p} \bigg)^2 . 
\end{align*}
Since $\scrB^\top \hw_H = \frac{\scrE_p(H)}{|\zv| |z-z_H|}$ per
$\req{optbias}$, the decomposition $BB^\top = \scrB \Lambda^2_p
\scrB^\top$, yields
\begin{align}
   |B^{\top} \hw|^2
&= \frac{|\Lambda_p \s \scrE_p(H) |^2}{|\zv|^2 |z-z_H|^2} 
+ 2 \frac{\ip{\Lambda^2_p\s \scrE_p(H)}{\scrB^\top v}}
{\rv_p \s |\zv|^2 |z - z_H| }  
+ \bigg( \frac{|\Lambda_p \scrB^\top v|}{|\zv| \rv_p} \bigg)^2 . 
\label{Bhw2}
\end{align}
Considering $\mf_p$, we examine $\hn_p^2
= \ip{\zv}{\hSig^{-1} \zv}$ for $p$ large.
Using $\req{eq:portcalc1}$ and $\req{zmzH}$,
\begin{align} \label{mhp2} 
\hn_p^2  = \ip{z}{\hSig^{-1} z} |\zv|^2 =
\bigg( \frac{|\zv|}{\hat{\sv}} \bigg)^2 \s[2]
    \Big( |z - z_H|^2 + \ip{z}{C_p} \Big)
\end{align}
where $C_p$, in  Lemma \ref{lem:hw}, was 
shown to have $\ip{z}{C_p}$ in $O(1 /\rv_p)$. 
We have $|\Lambda_p^2|$ in $O(\rv_p)$
for our modification of Assumption~\ref{asm:afm} and
assuming $|\zv|/\rv_p$ vanishes, 
\begin{align*}  
\hn_p^2 \ip{\hw}{\Sv \hw} 
&= \frac{ \ip{u_H}{\Sv u_H}}{\hat{\sv}^{2}} 
+ \frac{\ip{u_H}{\Sv u_H} \ip{z}{C_p}}{\hat{\sv}^2|z-z_H|^2 }
+ o_p(\Sv) + \frac{o_p(\Sv) \ip{z}{C_p}}{|z - z_H|^2 } 
\end{align*}
for $o_p(\Sv) = \frac{2 \ip{v}{\Sv u_H} |z-z_H|\s \rv_p 
+ \ip{v}{\Sv v}|z-z_H|^2}{\hat{\sv}^2 \rv_p^2}$
is in $O(1/\rv_p)$ as the eigenvalues of $\Sv_{p \times p}$ are bounded  
in $p$. So, the last three terms in the above display are in
$O(1/\rv_p)$.

Similarly, combining $\req{Bhw2}$ and $\req{mhp2}$, we obtain
\begin{align*}
\hn_p^2 |B^{\top} \hw|^2
&= \frac{|\Lambda_p \s \scrE_p(H) |^2}{\hat{\sv}^2} 
+ \frac{|\Lambda_p \s \scrE_p(H) |^2
\ip{z}{C_p}}{\hat{\sv}^2 |z-z_H|^2} 
+ o_p(B) + \frac{o_p(B) \ip{z}{C_p}}{|z - z_H|^2 } 
\end{align*}
where the 2nd term is in $O(|\scrE_p(H)|^2)$
and $o_p(B)$ is in $O(|\scrE_p(H)| + 1/\rv_p)$ as
\begin{align}
o_p(B) &=
\frac{2 \ip{\Lambda^2_p\s \scrE_p(H)}{\scrB^\top v}
|z-z_H|}{\rv_p \s \hat{\sv}^2 }  
+ \frac{1}{\rv_p} \bigg( \frac{|\Lambda_p \scrB^\top v| |z-z_H|}
{ \hat{\sv} \sqrt{\rv_p}} \bigg)^2 
\end{align}
where we note that $|\scrB|$ and $|v|$ are bounded in $p$.
The claim now follows.
\end{proof}

\section{Proofs for Section \ref{sec:pca}}
\label{app:pca}

Essential for our proofs is
Weyl's inequality for eigenvalue perturbations of a matrix
\citep{weyl1912}. In particular, for symmetric 
$m \times m$ matrices $A$ and $N$,
\[ \max_j  |\alpha_j - \alpha_j'| \le |N| \]
where $\alpha_j$ and $\alpha_j'$ denote the
$j$th largest eigenvalues of $A$ and $A + N$ respectively.

Define $\Lp = \Jc \Y^\top \Y\s \Jc \in \bbR^{n \times n}$ with $\Jc
= \Id - \frac{\s[2] \cn \cn^\top}{|\cn|^2}$ in 
$\req{Jc}$ and $\Y$ in $\req{data}$.
\begin{align} \label{Lp}
    \Lp = \Jc \X B^\top B \X^\top \Jc
  +  \Jc \E^\top \E  \Jc
  + \Jc \X  B^\top \E  \Jc
  + (\Jc \X  B^\top \E \Jc)^\top
\end{align}


By Assumption \ref{asm:afm}\ref{afm:BB} the following
$q \times q$ limit matrix exists, with the right side 
the eigenvalue decomposition with $q \times q$ orthogonal $\scrW$
and invertible, diagonal $\Lambda$.
\begin{align} \label{BBlim}
\limp \frac{B^\top B }{p}   = 
\scrW \Lambda^2 \scrW^\top
\end{align}

For $\sv^2$ of Assumption \ref{asm:asymp}\ref{asm:EE},
define the $n \times q$ matrix $\Xv$ and $n \times n$ 
matrix $L$ as,
\begin{align} \label{Mc}
L =  \Xv \Xv^\top +  \sv^2 \s \Jc  \s , 
\s[32]  \Xv = \Jc \X \scrW \Lambda \s  .
\end{align}

Let $\lambda^2_{j,n}(\Xv)$ denote the $j$th largest eigenvalue
of $\Xv \Xv^\top$  (also, $\Xv^\top \Xv$ for $j \le q$) associated with the $j$th column of
$\nu_{n\times n}(M)$, the eigenvectors of $\Xv \Xv^\top$. 
By Assumption
\ref{asm:asymp}\ref{asm:XX}, we have $\lambda^2_{j,n}(M) > 0$
for $j \le q$ and $\lambda^2_{j,n}(M) = 0$ otherwise. 

\begin{lemma}\label{lem:Sp}
Under Assumptions \ref{asm:afm} \& \ref{asm:asymp},
almost surely, $\limp W_p / p = L$ and,
\begin{align*}
 \limp \seig^2_{j,p} / p
= \left\{ \begin{array}{cc}
 (\lambda^2_{j,n}(M) + \sv^2) \s /n & 1 \le j < n \s , \\  
 \sv^2 /n & j = n, \s \Jc = \Id \s , \\ 
 0 &  \text{otherwise.}
\end{array} \right.
\end{align*}
\end{lemma}

As relevant for Assumption \ref{asm:asymp}\ref{asm:qnum}, above
implies $n_+$ (see $\req{snr-bulk}$) almost surely converges to
$n$ whenever $\Jc = \Id$ and $n_+$ converges to $n-1$ otherwise,
as $p \upto \infty$.

\vspace{0.16in}

\begin{proof}
We address the convergence of $\Lp/p$ as follows.  The sum of
the first and second terms in $\req{Lp}$ (scaled by $1/p$)
converge to $L$ due to Assumption \ref{asm:afm}\ref{afm:BB} and
Assumption \ref{asm:asymp}\ref{asm:EE}.  The last two terms in
$\req{Lp}$ (scaled by $1/p$) vanish by
Assumption~\ref{asm:asymp}\ref{asm:EB}. Since the nonzero
eigenvalues of $W_p$ are those of $\Y \Jc \Jc^\top \Y^\top = \Y
\Jc \Y^\top$, almost surely, $\seig^2_{j,p} / p$ converges to
the $j$th eigenvalue of $L$ by Weyl's inequality.

It remains to find the eigenvalues of $L$.  For $\Jc = \Id$, it
is easy to check that the eigenvalues of $L$ are just
$\lambda^2_{j,n}(\Xv) + \sv^2$ for $j = 1, \dots, n$ with
eigenvectors $\nu_{n \times n}(M)$.  When $\Jc \neq \Id$, we
have $L \cn = 0$ so that for any other eigenvector $v$ of $L$,
we have $\ip{v}{\cn} = 0$ and consequently $\Jc v = v$. It
follows when $\Jc \neq \Id$, the eigenvalues of $L$ are given by
$\lambda^2_{j,n}(\Xv) + \sv^2$ for  $j  < n$ and zero otherwise.
This concludes the proof.
\end{proof}

\begin{proof}[Proof of Theorem \ref{thm:pca}]
Taking \ref{pca:kappa} first, $\kappa_p^2$
in $\req{snr-bulk}$ is the average of 
$\seig^2_{j,p}$ for $q + 1 \le j \le  n_+$. By Lemma
\ref{lem:Sp}, for such $j$ we have $\seig^2_{j,p}/p
\to \sv^2 / n$ (i.e., $\lambda_{j,n}(\Xv) = 0$ for $j > q$).
Therefore, $\kappa_p^2/p \to \sv^2/n$ almost surely 
and part $\ref{pca:kappa}$ holds. 

Turning to part \ref{pca:Sp} we have $(\Psi \qeig_p)^2 =
\qeig_p^2 - \kappa^2_p \s \Id$ for $\Psi^2$ in $\req{snr-bulk}$.
By part \ref{pca:kappa} and Lemma \ref{lem:Sp}, $(\Psi
\qeig_p)^2_{jj}/p \to \lambda^2_{j,n}(\Xv)/n$ for $j \le q$.
For $K^2_p$ in \ref{pca:Sp}, since $(\np) (K^2_p)_{jj}$ is the $j$th largest eigenvalue of
$B\s \X^\top\Jc \X B^\top\s[-1] / p$ and equals  that of $\Jc \X
B^\top B \X\top \Jc\s / \s[-1] p$. The latter $n \times n$
matrix converges to $\Xv \Xv^\top$ by Assumption
\ref{asm:afm}\ref{afm:BB} and now, by Weyl's inequality, $ (\np)
(K^2_p)_{jj} \to \lambda^2_{j,n}(M)$ almost surely.  
Dividing by $n$ finishes the proof.

Henceforth, and in view of the above, we work
with the assumption that $\Y$ has rank $n_+ > q$ since for any
outcome there is a $p$ sufficiently large to ensure this.

For part \ref{pca:HBBH},  let $\Lw =\scrU \qeig_p^{-1} \sqrt{p/n}  $ 
where $\scrU$ is the $n \times q$ matrix of right singular
vectors of $\Y \Jc$ corresponding to its left singular 
vectors $\scrH = \nu_{p \times q}(\Y \Jc)$. We have,
\begin{align} \label{JcW}
 \Jc \Lw = \Lw \s, \s[32] 
 \Lw^\top \Lp \Lw =  p\s \Id
\end{align}
where the first identity is due to $\cn$ being a right singular vector 
of $\Y \Jc$ with value zero (i.e., $\scrU^\top g = 0_q$). 
The second identity comes from the singular value decomposition
which implies that 
$\Y\Jc \scrU /\s[-1] \sqrt{n} = \scrH \qeig_p$. The latter
further yields that,
\begin{align} \label{Hkey}
    \scrH =  \frac{\Y \Jc  \calW}{\sqrt{p}} 
=  \frac{\Y  \calW}{\sqrt{p}} 
    = \frac{1}{\sqrt{p}} B \X^\top  \Lw
    +\frac{1}{\sqrt{p}} \E  \Lw .
\end{align} 
Multiplying this by $BB^\dagger$ yields
that for $Z_p = B \X^\top  +  BB^{\dagger} \E$,
\begin{align} \label{BBH}
  BB^{\dagger}\scrH =  
   \frac{1}{\sqrt{p}} Z_p \Lw \s ,
    \s[32] 
\end{align}
and we expand on $Z_p Z_p^\top$ to obtain
(using that  $(BB^\dagger)^\top B = BB^\dagger B = B$),
\begin{align*}
 Z_p Z_p^\top 
  &=  \X B^\top B \X^\top + \E^\top B \X^\top + \X B^\top \E 
  + \E^\top BB^\dagger \E \\
  &=  \Y^\top \Y - \E^\top \E
  + \E^\top BB^\dagger \E
\end{align*}
Since the matrix $B$ is full rank,  $B^\dagger B = \Id$ 
and also $BB^\dagger = \scrB\scrB^\top$. Therefore,
\begin{align}
    \scrH^{\top}\scrB\scrB^{\top}\scrH 
  = \scrH^\top B B^{\dagger} \scrH 
  = \scrH^{\top} BB^{\dagger} BB^{\dagger} \scrH
  = (BB^{\dagger} \scrH)^{\top} (BB^{\dagger} \scrH) 
  \label{HBBH}
\end{align}
where at the last step we used that $BB^{\dagger}$ is symmetric. 

Combining $\req{BBH}$  and $\req{HBBH}$ with
$\req{JcW}$ with $\Lp$ in $\req{Lp}$ for which 
$\Lw^\top \Lw = \qeig_p^{-2} p/n$, adding and
subtracting $\hat{\sv}^2 \Id$ (where $\hat{\sv}^2
= n \kappa_p^2 /\s[-1] p$) and recalling that 
$\Psi = \Id_q - \kappa_p^2 \qeig_p^{-2}$,
\begin{align*}
  \scrH^\top \scrB \scrB^\top \scrH 
  & = \Lw^\top \big( \Lp - \E^\top \E 
    + \Jc \E^\top  BB^\dagger \E \Jc \big)  \Lw \s / p \\
  & = \Id_q + 
 \Lw^\top \big( \s \hat{\sv}^2 \Id - \hat{\sv}^2 \Id  
   - \E^\top \E /p \big) \Lw
    +  \Lw^\top\s[-2] \Jc \E^\top  BB^\dagger \E \Jc \Lw \s / p \\
  & = \Psi^2 + 
 \Lw^\top \big( \s \hat{\sv}^2 \Id  - \E^\top \E /p \big) \Lw 
+ \Lw^\top\s[-2] \Jc \E^\top  BB^\dagger \E\Jc\Lw \s / p \s . 
\end{align*}
From the above, we obtain the following bound.
\begin{align} \label{HBBHbound} 
| \scrH^\top \scrB \scrB^\top \scrH - \snr^2 |
\le \s |\Lw|^2\s \big(\s |\s\hat{\sv}^2 \Id  - \E^\top \E/p | 
 + | \Jc \E^\top BB^\dagger \E \Jc |\s \big)\s /p \s .
\end{align}
We have $|\scrW|^2 = |\scrU \qeig_p^{-1}|^2 \s (p/n)
\le \frac{|\scrU|^2}{\min_{j \le q} n \seig^{2}_{j,p}/p}$
and thus,
\begin{align} \label{lsupW}
\lsup |\scrW|^2 < \infty
\end{align}
almost surely
by Lemma \ref{lem:Sp} and using that $|\scrU|\le 1$.

By part \ref{pca:kappa} and Assumption
\ref{asm:asymp}\ref{asm:EE}, we also have that almost surely,
\begin{align} \label{Ink}
   \limp | \s \hat{\sv}^2 \Id   - \E^\top \E/p |  = 0 \s .
\end{align}
Since $|BB^\dagger \E\Jc |^2 
= | (B B^\dagger \E \Jc)^\top BB^\dagger \E \Jc |
= | \Jc \E^\top BB^\dagger \E \Jc |$,
it suffices to prove that
\begin{align}
    \limp 
 |BB^{\dagger}\calE \Jc |\s /\s[-1] \sqrt{p} = 0  \label{BBE}
\end{align}
almost surely.
In that regard,  we have
\begin{align*}
| BB^{\dagger}\calE \Jc \s |^2 / p
  &=  |\s \Jc \E^\top BB^\dagger \E \Jc\s | \s /p  
  \\&= |\s \Jc \E^\top B (B^{\top}B)^{-1} B^\top \E \Jc \s | \s /p  \\
  &=  |\s (p^{-1} \s[-1] B^\top B)^{-1/2} 
   B^\top \E \Jc \s |^2 / p^2   \s .
\end{align*}
Applying  Assumption \ref{asm:afm}\ref{afm:BB}, and 
in particular $\req{BBlim}$, yields
\[  \limp | BB^{\dagger}\calE \Jc |^2 / p
    \le | \Lambda^{-2} | \s  ( \limp |B^\top \E \Jc|\s /\s[-1] p )^2 \s .
\]
Assumption \ref{asm:asymp}\ref{asm:EB} and the fact that all
matrix norms on $\bbR^{q \times n}$ are equivalent concludes the
proof of $\req{BBE}$. Part \ref{pca:HBBH} now follows by
combining $\req{HBBHbound}$--$\req{BBE}$ and observing that each
$(\Psi^2)_{jj} = 1 - \kappa_p^2 / \seig^2_{j,p}$ for $j \le q$
is eventually in $(0,1)$ due to parts \ref{pca:Sp} and
\ref{pca:kappa}.

Lastly, for part \ref{pca:HBBz}, we again use that 
$BB^\dagger = \scrB\scrB^\top$ is symmetric,
that $\Jc \Lw  = \Lw$, and 
computing $\scrH^\top z$ from $\req{Hkey}$ while 
considering $\req{BBH}$  yields that almost surely
\begin{align*}
    \limp |\scrH^\top z -  \scrH \scrB \scrB^\top z | 
    &= \limp |\scrH^\top z - (BB^\dagger \scrH)^\top  z | \\
    &= \limp \frac{1}{\sqrt{p}} \big| 
   \Lw^\top \X B^\top z + \Lw^\top \E^\top z 
    - \Lw^\top Z_p^\top  z \big| \\
    &=  \limp \frac{1}{\sqrt{p}}
   |\s \Lw^\top \Jc \calE^\top z 
    - (BB^{\dagger}\E \Jc \Lw)^\top z | \\
    &\leq \limp |\Lw^\top |\s \Big( 
    \frac{1}{\sqrt{p}} |\Jc \E^\top z |
  + \frac{1}{\sqrt{p}} | BB^{\dagger}\E \Jc |\Big) = 0 
\end{align*}
by applying $\req{lsupW}$, $\req{BBE}$ and
Assumption \ref{asm:asymp}\ref{asm:Ee}. This concludes
the proof. \end{proof}

\begin{proof}[Proof of Theorem \ref{thm:pcabias}]
We first prove the norm of the numerator $
\scrB^\top(z - z_\scrH)$ of $\scrE_p(\scrH)$ in $\req{optbias_gps}$
converges to $|\Phi \scrH^\top z|$. Using that 
$z_\scrH = \scrH \scrH^\top z$ yields,
\begin{align}
|\scrB^\top (z - z_{\scrH})|^2 
&= \ip{\scrB^\top z}{ \scrB \scrB^\top z}
- 2 \ip{ \scrB^\top z}{ \scrB^\top z_\scrH}
+ \ip{\scrB^\top z_{\scrH}}{  \scrB^\top z_{\scrH}} \notag
\\&= | \scrB^\top z|^2
- 2 \ip{ \scrH^\top \scrB \scrB^\top z}{\scrH^\top z}
+ |\scrB^\top \scrH \scrH^\top z|^2 \s . \label{numerEH2}
\end{align}

Considering the last term in $\req{numerEH2}$, we obtain
\begin{align*}
|\scrB^\top \scrH \scrH^\top z|^2 
&= \ip{\scrH^\top z}{ (\scrH^\top \scrB  
\scrB^\top \scrH - \Psi^2)\s \scrH^\top z}
+ z^\top \scrH \Psi^2 \scrH^\top z \s .
\end{align*} 
and because
$|\ip{\scrH^\top z}{ (\scrH^\top \scrB  
\scrB^\top \scrH - \Psi^2)\s \scrH^\top z}|
\le |\scrH^\top z|^2  |\scrH^\top \scrB  
\scrB^\top \scrH - \Psi^2 |$
as well as $|\scrH^\top z| \le 1$,
we have by part \ref{pca:HBBH} of Theorem \ref{thm:pca} that
\begin{align} \label{numer1}
\limp |\scrB^\top \scrH \scrH^\top z|^2 
=  \limp |\Psi \scrH^\top z|^2 \s .
\end{align}

The second term in $\req{numerEH2}$ has,
\begin{align*}
  \ip{ \scrH^\top \scrB \scrB^\top z}{\scrH^\top z}
=   \ip{ \scrH^\top \scrB \scrB^\top z  - \scrH^\top z}{\scrH^\top z}
+ |\scrH^\top z|^2
\end{align*}
so that by part \ref{pca:HBBz} of Theorem \ref{thm:pca}, we have
\begin{align} \label{numer2}
\limp  \ip{ \scrH^\top \scrB \scrB^\top z}{\scrH^\top z}
= \limp |\scrH^\top z|^2
\end{align}

For the first term in $\req{numerEH2}$, due to
 Corollary \ref{cor:inv} and $\req{Hzlen}$
in particular,
\begin{align}
\limp |\scrB^\top z|^2 
= \limp |z_B|^2 = \limp | \Psi^{-1} \scrH^\top \scrB \scrB^\top z|
\end{align}
where we used that $|z_B| = |\scrB^\top z|$.
Since $|\Psi^{-1}| < \infty$ almost surely due to
part \ref{pca:HBBH} of Theorem \ref{thm:pca}, 
applying part \ref{pca:HBBz} of the same theorem now yields,
\begin{align} \label{numer3}
\limp |\scrB^\top z|^2 =
\limp | \Psi^{-1} \scrH^\top z|^2
\end{align}

Now, we rewrite the term $|\Phi \scrH^\top z|^2$ 
by substituting $\Phi = \snr^{-1} - \snr$ as,
\begin{align*}
    |\Phi\scrH^\top z|^2 
&=  z^\top \scrH (\Psi^{-1} - \Psi)
   (\Psi^{-1} - \Psi) \scrH^\top z \\ 
&= z^{\top} \scrH \Psi^{-2} \scrH^{\top} z 
- 2 z^{\top} \scrH \scrH^{\top} z 
+ z^{\top }\scrH \Psi^2 \scrH^\top z \nonumber\\
&= |\Psi^{-1} \scrH^{\top} z|^2 
- 2\s | \scrH^{\top} z |^2 + |\Psi \scrH^\top z|^2 \nonumber
\end{align*}
which confirms that $\limp
\big( |\scrB^\top (z - z_\scrH) | - |\Phi \scrH^\top z| \big)
= 0$ almost surely, and after taking the limit of  $\req{numerEH2}$
and substituting $\req{numer1}$, $\req{numer2}$ and
$\req{numer3}$. Lastly, from $\req{numer2}$,
\begin{align}\label{eq:2.6_Hzstayslessthan1}
    \lsup |\scrH^{\top}z| 
    &\le \lsup |\scrH^\top z| |\scrH^\top \scrB| |\scrB^\top z|
\leq \lsup  |\Psi^2|  | \scrB^\top z|
< \lsup | \scrB^\top z|
\end{align}
where we used that $|\scrH^\top \scrB|^2 = |\scrH^\top \scrB
\scrB^\top \scrH|$ and part \ref{pca:HBBH} of Theorem~\ref{thm:pca} 
which shows the limit of the latter
is $|\Psi^2|$ with every $\Psi^2_{jj}$ eventually in $(0,1)$.
From this, $\lsup |\Phi \scrH^\top z|
\le |\Phi| < \infty$ and $\sqrt{1 - |\scrH^\top z|^2}$ 
(denominator of $\scrE_p(\scrH)$ in $\req{optbias_gps}$)
is eventually
in $(0,1)$ almost surely.  We  now deduce that
$|\scrE_p(\scrH)| - \frac{|\Phi \scrH^\top z|}
{\sqrt{1 - |\scrH^\top z|^2 }}$ vanishes and 
$|\scrE_p(\scrH)|$
is eventually in $[0,\infty)$ almost surely. 
Lastly $|\scrE_p(\scrH)|$ converges to zero only if $\limp
\scrH^\top z = 0$ (when $|\varphi|$ vanishes) concluding 
the proof.

\end{proof}

\section{Proofs for Section \ref{sec:Hsharp}}
\label{app:sharp}

We begin with the following auxiliary result
which requires Assumption \ref{asm:asymp}.
As usual, $\scrB = \scrB_{p \times q} = \nu_{p \times q} (B)$ 
with $B = B_{p \times q}$ satisfying
Assumption \ref{asm:eB}. 

\begin{lemma} \label{lem:BHinv} The matrix
$\scrB^\top \scrH$ is eventually invertible almost surely.
\end{lemma}

\begin{proof} Considering the determinant of $\scrB^\top \scrH$,
almost surely,
\begin{align*}
  \limp ( \det \s (\scrB^\top \scrH) )^2 &
   = \limp \det \s (\scrH^\top \scrB\scrB^{\top}\scrH)  \\
    & = \limp \det  (\snr^2) 
\end{align*}
where we applied Theorem \ref{thm:pca}\ref{pca:HBBH} and the
continuity of the determinant. Moreover, the diagonal matrix
$\snr^2$ has every $\snr_{\jj}$ eventually in $(0,1)$ almost
surely.  Thus,  $\det(\scrB^\top \scrH)$ is almost
surely converging to a positive limit, i.e., $\scrB^\top \scrH$ is
eventually invertible. 
\end{proof}

\begin{proof}[Proof of Theorem \ref{thm:Hsharp}]
For $\scrH_z$ in $\req{Hz}$ and 
$\zpH, \hz$ in $\req{zpH}$, define
\begin{align} \label{Hplus}
\scrH_+ = \scrH \snr + \zpH \hz^\top
= \scrH_z  T_+\s ,
\s[32] T_+ =  \Big( \begin{array}{c}
  \snr \\ \hz^\top  \end{array} \Big) 
\end{align}
where $T_+ \in \bbR^{(q + 1) \times q}$ 
was first encountered in $\req{HsharpMat}$.
We compute,
\begin{align} \label{TpTp}
 \s[32] T_+^\top T_+ = \snr^2 + \hz \hz^\top
= \Rv \s \Phi^2 \Rv^\top \s ,
\s[32]  \Rv = \nu_{q \times q}(\snr^2 + \hz\hz^\top) \s ,
\end{align}
with the eigenvalue decomposition per $\req{Phi}$.
The singular value decomposition,
\begin{align} \label{Tplus} 
\s[32] T_+ = \scrT \Phi \ran^\top \s ,  
\s[32] \scrT =  \nu_{(q +1) \times q}(T_+) 
= ( \s \tau_1 \s[4] \cdots \tau_j \cdots \s[4] \tau_q \s ) \s , 
\end{align}
has $\tau_j \in \bbR^{q + 1}$ denoting the $j$th left singular vector
with $|\tau_j| = 1$ and value $\Phi_{\jj}$. We can write
the final estimator $\scrH_\sharp$ in $\req{Hsharp}$
in the form 
$\scrH_\sharp = \scrH_z T_\sharp $  (c.f. $\req{HsharpMat}$) where
\begin{align} \label{Tsharp}
 \s[32] T_\sharp = T_+ \scrM \Phi^{-1} = \scrT \s , 
\s[32] ( \scrH_\sharp = \scrH_z \scrT
= \scrH_+ \Rv \Phi^{-1}).
\end{align}

We prove the last part first. Using $\req{Hplus}$ and
multiplying from the right by $\scrB^\top$,
\begin{align*}
\scrB^\top \scrH_+ = (\scrB^\top \scrH) \snr
+ \scrE_p(\scrH) \hz^\top 
\end{align*}
where we used that $\scrB^\top \zpH = \scrE_p(\scrH)$. Applying
Corollary \ref{cor:inv}  yields,
\begin{align} \label{limBHplus1}
\limp \scrB^\top \scrH_+ = \limp 
\big( (\scrB^\top \scrH) \snr
+ (\scrB^\top \scrH 
\snr^{-2} \scrH^\top \scrB) \scrE_p(\scrH) \hz^\top \big)  \s .
\end{align}
Using the identity $\scrB^\top \zpH = \scrE_p(\scrH)$ and applying 
Theorem \ref{thm:pca} parts \ref{pca:HBBH}--\ref{pca:HBBz}, 
\begin{align} 
 \limp \snr^{-1} (\scrH^\top \scrB) \s \scrE_p(\scrH) 
  &=  \limp \snr^{-1} \frac{ (\scrH^\top \scrB) \scrB^\top z -
 (\scrH^\top \scrB\scrB^\top \scrH) \scrH^\top z }{|z-z_\scrH|} 
 \notag  
\\&=  \limp \snr^{-1}
 \frac{ (\Id - \snr^2) \scrH^\top z }{|z-z_\scrH|}  \notag
\\&=  \limp  \frac{ \Pd \scrH^\top z }{|z-z_\scrH|}  \label{PHBEH}
\end{align}
so that $\limp \snr^{-1} (\scrH^\top \scrB) \s \scrE_p(\scrH) 
= \limp \hz$ per $\req{zpH}$ with $\Pd = \snr^{-1}
- \snr$. This justifies the nontrivial part of 
the limit statement in $\req{HzBBH}$.
Continuing from $\req{limBHplus1}$, 
\begin{align*} 
\limp \scrB^\top \scrH_+ 
= \limp (\scrB^\top \scrH) \big( \snr
+ \snr^{-1} \hz \hz^\top \big) 
= \limp (\scrB^\top \scrH) \snr^{-1} (T_+^\top T_+) \s .
\end{align*}

Combining this with $\scrB^\top \scrH_\sharp = 
\scrB^\top \scrH_+ \Rv \Phi^{-1}$ 
per $\req{Tsharp}$ and $\req{TpTp}$
leads to the relation $\limp \scrB^\top \scrH_\sharp
= \limp (\scrB^\top \scrH) \snr^{-1} \Rv \Phi$. Therefore,
\[ \limp \scrH_\sharp^\top \scrB \scrB^\top \scrH_\sharp
= \limp \Phi \Rv^\top \snr^{-1} 
(\scrH^\top \scrB \scrB^\top \scrH) 
\snr^{-1} \Rv \Phi
\]
and the right side evaluates to $\limp \Phi^2$
by Theorem \ref{thm:pca}\ref{pca:HBBH}
and  that $\Rv^\top \Rv = \Id$.

Finally, we have $\scrH^\top_\sharp \scrH_\sharp = \scrT^\top
\scrH_z^\top \scrH_z \scrT = \Id$ using $\req{Tsharp}$ and the
fact that both matrices $\scrH_z$ and $\scrT$ have orthonormal
columns.

We now move to proving that 
$\scrE_p(\scrH_\sharp) \to 0$ for
$\scrH_\sharp = \scrH_z T_\sharp$ per $\req{Tsharp}$ with,
\begin{align} \label{EHsharp}
    \scrE_p(\scrH_\sharp) 
= \frac{\scrB^\top (z - z_{\scrH_{\sharp}})}
{\sqrt{1   - |z_{\scrH_{\sharp}}|^2}}
\end{align}
replacing $\scrH$ in $\req{optbias_gps}$ with $\scrH_\sharp$
and  applying $\req{zzH}$.
We prove the desired result in two steps below. In step 1 we 
show the denominator in $\req{EHsharp}$ is bounded away from zero
eventually.
In step 2 we prove that the numerator in
$\req{EHsharp}$ converges to zero.

\vspace{0.16in}

\textsc{Step 1}. We prove $|z_{\scrH_\sharp}|  < 1$
eventually in $p$ almost surely. 
Note that,
\begin{align*} 
|z_{\scrH_z T}|^2
&= |\scrH_z T (\scrH_z T)^\dagger z|^2 
= |\scrH_z T T^\dagger \scrH_z^\top z|^2 
= z^\top \scrH_z T T^\dagger
 \scrH_z^\top \scrH_z  T 
 T^\dagger \scrH_z^\top z
\\ &= z^\top \scrH_z T
 T^\dagger \scrH_z^\top z
\end{align*}
for any element $\scrH_z T$ in the family $\req{Tfamily}$
and where we have used that
$\scrH_z^\top \scrH_z = \Id$ and that $T^\dagger T = \Id$. 
Starting with $\req{Tplus}$ and $\req{Tsharp}$, we have  the spectral
decomposition
\begin{align} \label{zHsharp}
 T_\sharp T_\sharp^\dagger = \scrT \scrT^\dagger
= \scrT \scrT^\dagger = \tsum_{j=1}^q \tau_j \tau_j^\top \s .
\end{align}
using which and $\scrH_\sharp = \scrH_z T_\sharp$, we write
\begin{align} \label{zHsharp2} 
|z_{\scrH_\sharp}|^2 = 
| z_{\scrH_z T_\sharp}|^2 = 
z^\top \scrH_z T_\sharp T_\sharp^\top \scrH^\top_z z 
= \tsum_{j=1}^{q} \ip{\scrH^\top_z z}{\tau_j}^2 \s .
\end{align}

Next, since $\Big( \s[-3] \begin{array}{c}
  - \snr^{-1} \hz \\ 1 \end{array} \s[-3] \Big)^\top  T_+  
  = \hz^{\top} -\hz^{\top} \snr^{-1} \snr= 0_q$
with $T_+$ in $\req{Hplus}$, the vector
\begin{align*} 
\tau_{q + 1}  = \Big( \s[-2] \begin{array}{c}
  - \snr^{-1} \hz \\ 1 \end{array}  \s[-2] \Big) 
\frac{1}{\sqrt{1 + |\snr^{-1} \hz|^2}}
\end{align*}
is in the null space of $T_+$ and therefore in that of
$T_\sharp$ per $\req{Tsharp}$.  Since the column spaces of
$T_{\sharp}$ and $T_{\sharp}T_{\sharp}^{\dagger}$ are identical,
we have $\tau_1, \dots, \tau_q, \tau_{q+1} \in \bbR^{q+1}$ forms
a basis for $\bbR^{q+1}$.


Observing that $|\scrH_z^\top z|^2 
= |\scrH^\top z|^2 + |z - z_\scrH|^2 
= |\scrH^\top z|^2 + 1 - |z_\scrH|^2 = 1$, 
\[ 1 = |\scrH_z^\top z|^2 
= \tsum_{j=1}^{q+1} \ip{\scrH^\top_z z}{\tau_j}^2
= |z_{\scrH_\sharp}|^2 + \ip{\scrH_z^\top z}{\tau_{q + 1}}^2
\]
since $\tau_1, \dots, \tau_{q+1}$ forms
a basis for $\bbR^{q + 1}$ and applying $\req{zHsharp2}$.
Consequently,
\[
|z_{\scrH_\sharp}|^2  
= 1 - \ip{\scrH^\top_z z}{\tau_{q+1}}^2 \s .
\]
It now only suffices to show that
$\ip{\scrH^\top_z z}{\tau_{q+1}}^2 > 0$ eventually in $p$.
For $\hz$ in $\req{zpH}$,
\[ 
|\snr^{-1} \hz|^2 = \Big|
  \frac{ \snr^{-1} \Pd \scrH^\top z}
{|z - z_\scrH|}  \Big|^2
= \frac{ |(\snr^{-2} - \Id) \scrH^\top z|^2}
{|z - z_\scrH|^2 }, \]
and $\zpH$ in $\req{zpH}$ has
$\ip{\zpH}{z} = \frac{1 - \ip{z_\scrH}{z}}{|z - z_\scrH|}
= |z-z_\scrH|$ by $\req{zzH}$ and $\req{zmzH}$. Thus,
\begin{align}
\ip{\scrH^\top_z z}{\tau_{q+1}}^2
&= \bigg( \left( \begin{array}{c}
  \scrH^\top z \\ |z-z_\scrH|
\end{array} \right)^\top
\left( \begin{array}{c}
  - \Psi^{-1}\hz \\ 1
\end{array} \right) \bigg)^2 
\frac{1}{1+|\Psi^{-1}\hz|^2}  \notag \\
&= \frac{\big(|z-z_\scrH|^2 - z^\top \scrH 
(\Psi^{-2}-\Id) \scrH^\top z   \big)^2}
{|z-z_\scrH|^2 + |(\Psi^{-2}-I)\scrH^{\top}z|^2} \notag \\
&= \frac{\big(1-|z_{\scrH}|^2 + |z_{\scrH}|^2 
- z^\top \scrH \Psi^{-2} \scrH^\top z\big)^2}
{|z-z_\scrH|^2 + |(\snr^{-2} -\Id)\scrH^{\top}z|^2}\notag \\
&= \frac{\big( 1 - z^\top \scrH \Psi^{-2} \scrH^\top z \big)^2}
{|z-z_\scrH|^2 + |(\Psi^{-2}-I)\scrH^{\top}z|^2}.
\label{Hztau}
\end{align}
By $\req{Hzlen}$ and the fact that $|z_B| = |\scrB^\top z|
\le 1$ we have,
\begin{align*}
    \lsup z^\top \scrH \snr^{-2} \scrH^\top z
     = \lsup |\scrB^\top z|^2
\end{align*}
under Assumption \ref{asm:eB} which also guarantees 
$\lsup |\scrB^\top z| < 1$. We deduce that the numerator of
$\req{Hztau}$ is eventually strictly
positive almost surely. The denominator is finite as
$|z - z_\scrH|^2 \le 1$ and the eigenvalues of 
$\Psi^{-2}$ are finite by  Theorem~\ref{thm:pca}\ref{pca:HBBH}.

Thus, 
$\ip{\scrH^\top_z z}{\tau_{q+1}}^2$ is almost surely bounded
away from zero eventually.

\vspace{0.16in}

\textsc{Step 2}. We prove that the numerator of 
$\req{EHsharp}$ almost surely eventually vanishes. We
omit the \tq{almost surely} clause for brevity below.
Recall that $\req{slick}$
supplies that
\begin{align} \label{Numzero}
  \scrB^\top (z - z_{\scrH_z T_*}) = 0 
\end{align}
provided $T^\top_* T_*$ is invertible for
$T_* = \scrH_z^\top \scrB$. To establish the latter,
we directly compute $T_*^\top = \big( \begin{array}{cc}
\scrB^\top \scrH & \scrE_p(\scrH) 
\end{array} \big)$
using $\scrB^\top \zpH = \scrE_p(\scrH)$
with $\scrE_p(\scrH)$ in $\req{optbias_gps}$. Then,
\begin{align} \label{TsTs}
T^\top_* T_* = \scrB^\top \scrH_z \scrH_z^\top \scrB
= \scrB^\top \scrH\scrH^\top \scrB
+ \scrE_p(\scrH)\s \scrE^\top_p(\scrH) 
\end{align}
and we deduce that $|T^\top_* T_*|$ is eventually bounded
since the columns of $\scrB$ and $\scrH$ have unit length
and $|\scrE_p(\scrH)|$ is eventually finite by Theorem
\ref{thm:pcabias}.
Since both terms in $\req{TsTs}$ are positive semidefinite 
and $\scrB^\top \scrH$ eventually invertible by Lemma
\ref{lem:BHinv}, all eigenvalues of $\req{TsTs}$ are 
strictly positive. Hence, $T_*^\top T_*$  is eventually
invertible.

In view of $\req{Numzero}$, it only suffices to prove that the
difference between $z_{\scrH_\sharp}= z_{\scrH_zT_{\sharp}}$ and
$z_{\scrH_zT_{*}}$ vanishes in some norm. By
Lemma \ref{lem:K} with $K = \scrB^\top \scrH \snr^{-1} \Rv
\Phi^{-1}$,
\begin{align}\label{limproj}
    \lsup 
   |z_{\scrH_zT_{\sharp}}-z_{\scrH_zT_{*}}| 
  = \lsup
|z_{\scrH_zT_{\sharp}}- z_{\scrH_zT_{*}K}| \s ,
\end{align}
owing to Lemma \ref{lem:BHinv} and 
Theorem \ref{thm:pca}\ref{pca:HBBH} which guarantee
that $\scrB^\top \scrH$ and
$\snr^{-1} \Rv \Phi^{-1}$ (and hence $K$) are eventually
invertible. Substituting $T_\sharp = T_+ \Rv \Phi^{-1}$,
we have
\begin{align*}
\limp | T_\sharp - T_* K |
&\le \limp \Big| \Big( \begin{array}{c}
  \snr \\ \hz^\top  \end{array} \Big) 
 - \scrH_z^\top \scrB \scrB^\top \scrH \snr^{-1} \Big| 
\s |\Rv \Phi^{-1}| = 0 
\end{align*}
which confirms $\req{Tprime}$ and applies $\req{HzBBH}$ which
was justified above (see $\req{PHBEH}$).
 Since the mapping 
$T \rightarrow z_{\scrH_z T}$ from the domain of
real $(q + 1)\times q$, full column rank matrices is continuous, 
we have via $\req{limproj}$ that 
$\lim\limits_{p\rightarrow\infty}|z_{\scrH_zT_{\sharp}}-z_{\scrH_zT_{*}}|
= 0$ as required.

We remark that since  
$\lsup |z_{\scrH_z T_\sharp}| < 1$, 
we now have $\lsup |z_{\scrH_z T_*}| < 1$.
This also proves that
$\scrE_p(\scrH_z T_*) = 0$ eventually in $p$ (see
comments below $\req{slick}$).

\end{proof}

\section{The Eigenvector Selection Function}
\label{app:select}

For any matrix $A\in\mathbb{R}^{m\times \ell}$, we enumerate
\(q\leq \min(\ell,m)\) singular values (in descending order) and
their left singular vectors in a well defined way. We start by
ordering all $d$ distinct singular values of $A$ as \(\lambda_1
> \lambda_2 > \cdots > \lambda_d\) (c.f.  $\Lambda_{\jj}$ in
Section~\ref{sec:notation}) and uniquely identifying the linear
subspaces \(\calK_1, \calK_2, \dots, \calK_d\) formed by the
associated left singular vectors.  Given the first $k-1$ left
singular vectors $v_1,v_2,..,v_{k-1}$ are selected, we select
the \(k\)th left singular vector \(v_k\) by taking the following
steps.
\begin{enumerate}
    \item Identify the unique \(j\in\{1,2,\dots,d\}\) for which,
    \begin{align}
        k\in \Big(\sum_{a=1}^{j-1}\dim(\mathcal{K}_{a}), 
  \sum_{a=1}^{j}\dim(\calK_a) \Big] \s . \label{eq:definitionofj}
    \end{align}
    \item Let \(\calK_{j,k} \subseteq \calK_j\) denote the
orthogonal complement of the subspace formed by the subset of
vectors \(v_1, \dots, v_{k-1}\) corresponding to \(\lambda_j\),
where the orthogonal complement is taken within
\(\mathcal{K}_j\). For the standard basis elements 
\(\rme_1, \rme_2, \dots, \rme_m\) of \(\mathbb{R}^m\), we
identify the unique \(\rme_s\) as the first one that is not 
orthogonal to \(\calK_{j,k}\). 

\item Set $v_k$ as the orthogonal projection of $\rme_s$ onto
$\calK_{j,k}$ normalized to $|v_k|=1$.

\end{enumerate}

Implementing this process sequentially on $k=1,2,...,m$
assembles a list of left singular vectors \(v_1, \dots, v_m\)
with associated singular values in decreasing order.  We define
$\nu_{ m \times q}(A)$ as an $m\times q$ matrix carrying
$v_1,v_2,..,v_q$ at its columns.

\begin{remark}
In the second step to define the \( k \)th left singular
vector, note that the subspace \(\mathcal{K}_{j,k}\) is of
non-zero dimension by $\req{eq:definitionofj}$. Moreover, 
the uniquely defined standard basis element \( \rme_s \) has
to exist. If it does not exist, the whole space \( \mathbb{R}^m
\) becomes orthogonal to \(\mathcal{K}_{j,k}\), implying that
\(\mathcal{K}_{j,k}\) is of zero dimension, which contradicts
our previous assertion.
\end{remark}

\begin{example}
We illustrate the above procedure with \(A = \Id_m\).  The
matrix \(\Id_m\) has \(\lambda_1 = 1\) as the sole singular
value which corresponds to the subspace of left singular vectors
\(\calK_1 = \bbR^m\). For \(k=1\) in the algorithm introduced
above, we obtain the corresponding \(j\) determined as $1$ by
$\req{eq:definitionofj}$. The subspace \(\calK_{j,k} =
\calK_{1,1}\) equals \(\calK_1 = \mathbb{R}^m\) as there has not
been any selection yet. Then the first of \(\rme_1, \dots,
\rme_m\) that is not orthogonal to \(\calK_{1,1} =
\mathbb{R}^m\) would clearly be \(\rme_1\).  Hence, \(v_1 =
\rme_1\) is the normalized orthogonal projection of \(\rme_1\)
onto \(\calK_{1,1} = \mathbb{R}^m\).  Next, we assume as an
induction hypothesis that \(v_1 = \rme_1, v_2 = \rme_2, \dots,
v_{k-1} = \rme_{k-1}\) and implement the \(k\)th step of the
algorithm to show \(v_k = e_k\). Clearly, \(j\) defined by
$\req{eq:definitionofj}$ corresponding to \(k\) is $1$.
Moreover, \(\calK_{1,k}\), the orthogonal complement of the
subspace formed by the vectors previously selected for the
singular value \(\lambda_1 = 1\), is spanned by \(\rme_k,
\rme_{k+1}, \dots, \rme_m\). Hence, the first of \(\rme_1,
\dots, \rme_m\) that is not orthogonal to \(\calK_{1,k}\) is
\(\rme_k\). That sets \(v_k = \rme_k\). As a result, we obtain
\(\nu_{m \times q}(\Id_m)\) assembled as \([\rme_1, \rme_2,
\dots, \rme_q]\) so that its \(i\)th column contains the
coordinate vector \(\rme_i\).
\end{example}

\section{Capon Beamforming}
\label{app:lit}

One important illustration of the pathological behaviour
described below $\req{Qhx}$ concerns robust (Capon) beamforming
(see \ci{cox1987}, \ci{li2005}) and \ci{vorobyov2013}).  Some
recent work that applies spectral methods  for robust
beamforming may be found in \ci{zhu2020}, \ci{luo2023} and
\ci{chen2024}, who survey related work. The importance of the
covariance estimation aspect of robust beamforming is also
well-recognized (e.g., \ci{stoica2007}, \ci{chen2010} and
\ci{xie2021}). In particular, the LW shrinkage estimator
developed in \ci{ledoit2004a} has had noteworthy influence on this
literature, despite being originally proposed for portfolio
selection in finance \cite{ledoit2003}. Typical application of
this estimator employs the identity matrix as the \tq{shrinkage
target}, which leaves the eigenvectors of the sample covariance
matrix unchanged (fn.  \ref{lwfoot}). However, the estimation
error in the sample  eigenvectors (especially for small
sample/snapshot sizes as is our setting) is known to have
material impact \cite{cox2002}. One (rare) example of robust
beamforming work that attempts to \tq{de-noise} sample
eigenvectors directly is  \ci{qj2015}.  But, their analysis does
not overlap with our $\req{obf}$--$\req{EHslim}$.


\end{appendix}

\begin{acks}[Acknowledgments]
We thank Haim Bar and Alec Kercheval for useful feedback on the
motivating optimization example (Section \ref{sec:motive}) in
our introduction. We thank Lisa Goldberg for an insightful
discussion that led to the main theorem of Section
\ref{sec:impossible}. We thank Kay Giesecke for many helpful
comments on an earlier draft of this paper.  We thank the
participants of the 2023 SIAM Conference on Financial
Mathematics in Philadelphia, PA, the UCLA Seminar on Financial
and Actuarial Mathematics, Los Angeles, CA, the CDAR Risk
Seminar, UC Berkeley and the AFT Lab Seminar, Stanford CA for
their comments and feedback on many of the ideas that led to
this manuscript.
\end{acks}
\bibliographystyle{imsart-nameyear} 
\bibliography{bibliography}       


\end{document}